\documentclass[12pt]{article}
\usepackage{amsmath,amssymb,amsthm}
\usepackage{bm} 
\usepackage{ascmac} 
\usepackage{amscd} 
\usepackage{stmaryrd} 

\newtheorem{thm}{Theorem}
\newtheorem{crl}{Corollary}
\newtheorem{cnj}{Conjecture}
\newtheorem{prp}{Proposition}
\newtheorem{lmm}{Lemma}
\newtheorem{rmk}{Remark}
\newtheorem{dfn}{Definition}

\newcommand{\R}{\Rightarrow}

\newcommand{\LR}{\Leftrightarrow}
\newcommand{\st}{\stackrel}
\newcommand{\ra}{\rightarrow}
\newcommand{\mt}{\mapsto}

\newcommand{\HFR}{\mathrm{Hom}(F,\mathbb R)}
\newcommand{\HFCp}{\mathrm{Hom}(F,\mathbb C_p)}

\newcommand{\prn}{\mathbb R_{+}}

\newcommand{\nni}{\mathbb Z_{\geq 0}}
\newcommand{\ttt}{\,{}^t}

\newcommand{\I}{\rm (I)}
\newcommand{\II}{\rm (I\hspace{-.1em}I)}
\newcommand{\III}{\rm (I\hspace{-.1em}I\hspace{-.1em}I)}

\newcommand{\mfpi}{\mathfrak p_\iota}
\newcommand{\fpi}{\mathfrak f\mathfrak p_\iota}

\makeatletter
\newcommand{\subjclass}[2][2010]{%
  \let\@oldtitle\@title%
  \gdef\@title{\@oldtitle\footnotetext{#1 \emph{Mathematics subject classification(s).} #2}}%
}
\newcommand{\keywords}[1]{%
  \let\@@oldtitle\@title%
  \gdef\@title{\@@oldtitle\footnotetext{\emph{Key words and phrases.} #1.}}%
}
\makeatother

\title{On the ratios of Barnes' multiple gamma functions to the $p$-adic analogues}
\author{Tomokazu Kashio\thanks{Tokyo University of Science, \texttt{kashio\_tomokazu@ma.noda.tus.ac.jp}}}

\begin{document}


\maketitle

\begin{abstract}
Let $F$ be a totally real field.
For each ideal class $c$ of $F$ and each real embedding $\iota$ of $F$,
Hiroyuki Yoshida defined an invariant $X(c,\iota)$ as a finite sum of log of Barnes' multiple gamma functions with some correction terms.
Then the derivative value of the partial zeta function $\zeta(s,c)$ has a canonical decomposition 
$\zeta'(0,c)=\sum_{\iota}X(c,\iota)$, where $\iota$ runs over all real embeddings of $F$.
Yoshida studied the relation between $\exp(X(c,\iota))$'s, Stark units, and Shimura's period symbol.
Yoshida and the author also defined and studied the $p$-adic analogue $X_p(c,\iota)$:
In particular, we discussed the relation between the ratios $[\exp(X(c,\iota)):\exp_p(X_p(c,\iota))]$ and Gross-Stark units.
In a previous paper, the author proved the algebraicity of some products of $\exp(X(c,\iota))$'s.
In this paper, we prove its $p$-adic analogue.
Then, by using these algebraicity properties, we discuss the relation between the ratios $[\exp(X(c,\iota)):\exp_p(X_p(c,\iota))]$ and Stark units.
\end{abstract}


\section{Introduction}

Let $F$ be a totally real field, $\mathcal O_F$ the ring of integers of $F$, $\mathfrak f$ an integral ideal of $F$, 
$I_{\mathfrak f}$ the group of all fractional ideals of $F$ relatively prime to $\mathfrak f$.
We consider the narrow ideal class group modulo $\mathfrak f$ defined as 
\begin{align*}
C_{\mathfrak f}:=I_{\mathfrak f}/\{(\alpha) \in I_{\mathfrak f} \mid \alpha \in F^\times,\ \alpha \equiv 1\ {\bmod}^*\ \mathfrak f,\ \alpha>\!\!>0\}.
\end{align*}
Here $\alpha>\!\!> 0$ means that $\alpha$ is totally positive.
We denote the class in $C_\mathfrak f$ of a fractional ideal $\mathfrak a$ by $[\mathfrak a]$.
Shintani \cite{Shin2} gave an explicit formula for the derivative value $\zeta'(0,c)$ of 
the partial zeta function $\zeta(s,c):=\sum_{\mathcal O_F \supset \mathfrak a \in c}N\mathfrak a^{-s}$ with $c \in C_\mathfrak f$.
The main term is a finite sum of $\log$ of Barnes' multiple gamma functions $\Gamma(z,\bm v)$ (Definition \ref{mgf}) of the following form:
\begin{align} \label{Sf}
\zeta'(0,c)=\sum_{\iota \in \HFR} \left(\sum_{j \in J} \sum_{\bm x \in R(c,\bm v_j)} 
\log(\Gamma(\iota(\bm x\ttt\bm v_j),\iota(\bm v_j)))\right) + \text{ correction terms}. 
\end{align}
Here $\bm v_j$ are suitable vectors whose entries are totally positive integers of $F$, 
$R(c,\bm v_j)$ are finite sets of vectors whose entries are positive rational numbers (for details, see Theorem \ref{Sd}, Definition \ref{R}).
We denote by $\bm x\ttt\bm v_j$ the inner product, by $\HFR$ the set of all real embeddings of $F$.
Yoshida \cite[Chap.\ III, (3.6)--(3.9)]{Yo} discovered an appropriate decomposition $\sum_{\iota \in \HFR} W(c,\iota)+V(c,\iota)$ of the ``correction terms'' 
and defined the class invariant 
\begin{align*}
X(c,\iota):=\sum_{j \in J} \sum_{\bm x \in R(c,\bm v_j)} \log(\Gamma(\iota(\bm x\ttt\bm v_j),\iota(\bm v_j))) +W(c,\iota)+V(c,\iota).
\end{align*}
It follows that we have
\begin{align*}
\zeta'(0,c)=\sum_{\iota \in \HFR} X(c,\iota).
\end{align*}
Since Shintani's expression (\ref{Sf}) is not unique,
Yoshida's invariant has an ambiguity: $\exp(X(c,\iota))$ is well-defined up to multiplication by a rational power of a unit of $\iota(F)$.
Yoshida conjectured that each value of Shimura's period symbol $p_K$ (\cite[Theorem 32.5]{Shim}) can be expressed 
as a product of rational powers of $\exp(X(c,\iota))$'s \cite[Chap.\ III, Conjecture 3.9]{Yo}.
Yoshida also studied the relation between Stark units and $\exp(X(c,\iota))$'s.
On the other hand, Yoshida and the author defined and studied the $p$-adic analogue $X_p(c,\iota)$ in \cite{KY1,KY2}.
In particular we formulated a refinement of a $p$-adic analogue of the Stark conjecture by Gross.
Let us explain more precisely:
Let $K/F$ be an abelian extension of number fields, $S$ a finite set of places of $F$. We assume that
\begin{itemize}
\item $S$ contains all infinite places of $F$ and all ramified places in $K/F$.
\item $S$ contains a distinguished place $v$ which splits completely in $K/F$.
\item $\#S>2$. (This assumption is for simplicity. $\#S>1$ is essential.)
\end{itemize}
We fix a place $w$ of $K$ dividing $v$. 
Then the rank $1$ abelian Stark conjecture, which is a version of Stark's conjectures in \cite{St}, 
implies the following statement.

\begin{cnj}
There exists a $v$-unit $u\in K$, which is called a Stark unit, satisfying 
\begin{align*}
\log \parallel\!\! u^{\tau}\!\!\parallel_{w} =-W_K\zeta'_{S}(0,\tau) \quad (\tau\in \mathrm{Gal}(K/F)).
\end{align*}
\end{cnj}

Here $W_K$ is the number of roots of unity in $K$. We define the partial zeta function associated with $S,\tau$ by
\begin{align*}
\zeta_{S}(s,\tau):=\sum_{(\frac{K/F}{\mathfrak{a}})=\tau,\ (\mathfrak{a},S)=1}N\mathfrak{a}^{-s},
\end{align*}
where $\mathfrak{a}$ runs over all integral ideals of $F$, relatively prime to each finite place in $S$, 
whose images under the Artin symbol $(\frac{K/F}{*})$ equal $\tau$.
If $w$ is a finite place, we put $\parallel\!\! x\!\!\parallel_{w}:=Nw^{-\mathrm{ord}_{w}(x)}$.
Otherwise we take an embedding $\iota_w\colon K\hookrightarrow \mathbb{C}$ associated with $w$ and
put $\parallel\!\! x\!\!\parallel _{w}:=|\iota_w(x)|$ or $|\iota_w(x)|^2$ 
for a real or complex place $w$ respectively.
Note that when $v$ is a real place, assuming the Stark conjecture, we can write a Stark unit explicitly as 
\begin{align} \label{Su}
\exp(2\zeta'_{S}(0,\tau)) \in \iota_w(\mathcal O_K^\times).
\end{align}
Here we see that $W_K=2$ since $K$ has a real place $w$. 

Next we consider the case when $v$ is a finite place lying above a rational prime $p$.
We additionally assume that 
\begin{itemize}
\item $F$ is a totally real field, $K$ is a CM-field.
\item $S$ contains all places of $F$ lying above $p$.
\end{itemize}
Throughout this paper, we regard each number field as a subfield of $\overline{\mathbb{Q}}$
and fix two embeddings $\overline{\mathbb{Q}}\hookrightarrow \mathbb{C}$, $\overline{\mathbb{Q}}\hookrightarrow \mathbb{C}_{p}$.
Here we denote by $\mathbb{C}_{p}$ the $p$-adic completion of the algebraic closure $\overline{\mathbb{Q}_{p}}$ of $\mathbb{Q}_{p}$.
For a number field $L$, we denote by $\mathfrak{p}_{L}:=\mathfrak{p}_{L,\mathrm{id}}$ the prime ideal corresponding to the $p$-adic topology on $L$ induced by 
$\mathrm{id} \colon L \hookrightarrow \overline{\mathbb Q} \hookrightarrow \mathbb C_p$.
We assume that 
\begin{itemize}
\item $v=\mathfrak{p}_{F}$, $w=\mathfrak{p}_{K}$.
\end{itemize}
Put $R:=S-\{\mathfrak{p}_{F}\}$. 
We see that $\zeta_{S}(s,\tau)=(1-N\mathfrak{p}_{F}{}^{-s})\zeta_{R}(s,\tau)$ since $\mathfrak{p}_{F}$ splits completely in $K$.
It follows that 
\begin{align*}
\zeta_{S}'(0,\tau)=\zeta_{R}(0,\tau) \log N\mathfrak{p}_{F}.
\end{align*}
Hence, by $\zeta_{R}(0,\tau) \in \mathbb Q$, there exist a positive integer $W$ and a $\mathfrak{p}_{F}$-unit $u'\in K$ which satisfy
\begin{align} \label{GSu}
\log \parallel\!\! u'^{\tau}\!\!\parallel_{w} =-W\zeta'_{S}(0,\tau) \quad (\tau\in \mathrm{Gal}(K/F)).
\end{align}
The Stark conjecture with $v=\mathfrak p_F$ states that we can take $W=W_K$. 
Gross conjectured the following property of $u'$ as a $p$-adic number.

\begin{cnj}[{\cite[Conjecture 3.13]{Gr2}}] \label{GSc}
Let $u'\in K$ be a $\mathfrak{p}_{F}$-unit satisfying (\ref{GSu}). Then we have 
\begin{align*}
\log_{p} N_{K_{\mathfrak p_K}/\mathbb Q_p}({u'}^\tau)
=-W\zeta'_{p,S}(0,\tau) \quad (\tau\in \mathrm{Gal}(K/F)).
\end{align*}
\end{cnj}

Here $\log_{p}$ denotes Iwasawa's $p$-adic $\log$ function and $\zeta_{p,S}(s,\tau)$ denotes the $p$-adic interpolation function for $\zeta_{S}(s,\tau)$. 
Dasgupta-Darmon-Pollack \cite{DDP} proved that Conjecture \ref{GSc} holds true under certain conditions.
In \cite{KY1}, we formulated a conjecture which expresses $\log_p {u'}^{\tau}$ (without the norm $N_{K_{\mathfrak{p}_{K}}/\mathbb{Q}_{p}}$)
by using not only $X_p(c,\iota)$ but also $X(c,\iota)$.
More precisely, we see that the ratio $[\exp(X(c,\iota)):\exp_p(X_p(c,\iota))]$ is well-defined up to roots of unity (Corollary \ref{crlofmain1}-(i)),
and our conjecture (Conjecture \ref{cnjky1}) states that 
$u'$ can be expressed as a product of $\frac{\exp(X(c,\iota))}{\exp_p(X_p(c,\iota))}$'s, up to $\ker \log_p$ (Proposition \ref{eqofcnjs}-(ii)).
One of the main results in this paper states that we can express a Stark unit as another product of $\frac{\exp(X(c,\iota))}{\exp_p(X_p(c,\iota))}$'s
(Corollary \ref{crlofmain1}-(ii)).
Roughly speaking, 
Stark units (\ref{Su}) and Gross-Stark units (\ref{GSu}) can be written 
as a product of $\frac{\exp(X(c,\iota))}{\exp_p(X_p(c,\iota))}$'s uniformly (Remark \ref{final}).

To clarify the meaning of the main results in this paper, we briefly describe the results in \cite{Ka2}.
Let $F=\mathbb Q$, $v$ the unique real place of $\mathbb Q$.
Then the Stark conjecture implies the following ``reciprocity law'' on the sine function (in other words, on cyclotomic units):
\begin{align} \label{rlos}
\sin(\tfrac{a}{m}\pi)^\tau=\pm \sin(\tau(\tfrac{a}{m})\pi) \quad (\tfrac{a}{m} \in \mathbb Q\cap (0,1),\ \tau \in \mathrm{Aut}(\mathbb C)).
\end{align}
Here we define $\tau(\tfrac{a}{m}) \in \mathbb Q\cap (0,1)$ by $\tau(\zeta_m^a)=\zeta_m^{\tau(\frac{a}{m})m}$ with $\zeta_m:=e^{\frac{2 \pi i}{m}}$.
In \cite{Ka2} we proved a reciprocity law (\ref{rlonb}) on a period-ring valued beta function, which is a refinement of (\ref{rlos}) in the following sense: 
For simplicity, we assume that $p\mid m$, $p\neq 2$ here.
We use the following notation.
\begin{align*}
&[a:b]  [c:d] :=[ac:bd], \\ 
&[a:b] = [c:d] \text{ means } \frac{a}{c} =\frac{b}{d}, \\
&[a:b] \equiv [c:d] \bmod \mu_\infty \text{ means } \frac{a}{c} \frac{d}{b} \in \mu_\infty,
\end{align*}
where $\mu_\infty$ denotes the group of all roots of unity.
\begin{itemize}
\item Let $\Gamma_p \colon \mathbb Q_p \ra \mathbb C_p^\times $ be a generalization \cite[Lemma 4.2]{Ka2} of Morita's $p$-adic gamma function.
When $F=\mathbb Q$, $\mathfrak f=(m)$, $c=[(a)]\in C_{(m)}$ $(a \in \mathbb N,\ (a,m)=1)$,
we see that 
\begin{align*}
[\exp(X([(a)],\mathrm{id})):\exp_p(X_p([(a)],\mathrm{id}))]
\equiv [\Gamma(\tfrac{a}{m})/\sqrt{2\pi}:\Gamma_p(\tfrac{a}{m})] \bmod \mu_\infty.
\end{align*}
Moreover we may regard this ratio as a refinement of a cyclotomic unit:
\begin{align*}
[\Gamma(\tfrac{a}{m})/\sqrt{2\pi}:\Gamma_p(\tfrac{a}{m})][\Gamma(\tfrac{m-a}{m})/\sqrt{2\pi}:\Gamma_p(\tfrac{m-a}{m})] 
\equiv [1:2\sin(\tfrac{a}{m}\pi)] \bmod \mu_\infty.
\end{align*}
\item We consider the $m$th Fermat curve $F_m\colon x^m+y^m=1$ 
and its differential forms $\eta_{\frac{a}{m},\frac{b}{m}}:=x^{a-1}y^{b-m}\mathit{d}x$ ($0<a,b<m$, $a+b\neq m$).
Although the period integral $\int_\gamma \eta_{\frac{a}{m},\frac{b}{m}}$ 
depends on the choice of a closed path $\gamma \subset F_m(\mathbb C)$,
the ratio $[\int_\gamma \eta_{\frac{a}{m},\frac{b}{m}}:\int_{p,\gamma} \eta_{\frac{a}{m},\frac{b}{m}}]$ is constant.
Here $\int_{p,\gamma} \eta_{\frac{a}{m},\frac{b}{m}} \in B_{\mathrm{dR}}$ is the $p$-adic period defined by the $p$-adic Hodge theory, 
$B_{\mathrm{dR}}$ is Fontaine's $p$-adic period ring.
Moreover we have
\begin{align*} \textstyle
[\int_\gamma \eta_{\frac{a}{m},\frac{b}{m}}:\int_{p,\gamma} \eta_{\frac{a}{m},\frac{b}{m}}]
[\int_\gamma \eta_{\frac{m-a}{m},\frac{m-b}{m}}:\int_{p,\gamma} \eta_{\frac{m-a}{m},\frac{m-b}{m}}]=[2 \pi i : (2 \pi i)_p],
\end{align*}
where $(2 \pi i)_p \in B_{\mathrm{dR}}$ is the $p$-adic counterpart of $2 \pi i$. 
\item Rohrlich's formula in \cite{Gr1} implies 
that $\frac{\int_\gamma \eta_{\frac{a}{m},\frac{b}{m}}}{\frac{\Gamma(\frac{a}{m})\Gamma(\frac{b}{m})}{\Gamma(\frac{a+b}{m})}} \in \overline{\mathbb Q}$.
It follows that 
\begin{align*}
\mathfrak B(\tfrac{a}{m},\tfrac{b}{m})
:=\frac{\frac{\Gamma(\frac{a}{m})\Gamma(\frac{b}{m})}{\Gamma(\frac{a+b}{m})}}{\int_{\gamma} \eta_{\frac{a}{m},\frac{b}{m}}}
\frac{\int_{p,\gamma} \eta_{\frac{a}{m},\frac{b}{m}}}{\frac{\Gamma_p(\frac{a}{m})\Gamma_p(\frac{b}{m})}{\Gamma_p(\frac{a+b}{m})}} 
\in B_{\mathrm{dR}}
\quad  (p \nmid ab(a+b))
\end{align*}
are well-defined up to $\mu_\infty$.
\item Coleman's formula in \cite{Co} implies that \cite[Theorem 7.2-(ii)]{Ka2}
\begin{align} \label{rlonb}
\Phi_\tau (\mathfrak B(\tfrac{a}{m},\tfrac{b}{m})) \equiv p^{\frac{\deg \tau}{2}} \mathfrak B(\tau(\tfrac{a}{m}),\tau(\tfrac{b}{m})) \bmod \mu_\infty 
\quad (\tau \in W_p).
\end{align}
Here we denote by $W_p \subset \mathrm{Gal}(\overline{\mathbb Q_p}/\mathbb Q_p)$ the Weil group, 
by $\Phi_\tau$ the absolute Frobenius automorphism associated with $\tau \in W_p$, acting on a subring of $B_\mathrm{dR}$.
\item We can express $2\sin(\tfrac{a}{m}\pi)$ as a product of rational powers of $\mathfrak B(\tfrac{a}{m},\tfrac{b}{m})$'s
up to $\mu_\infty,(2 \pi i)_p$.
Hence the reciprocity law (\ref{rlonb}) implies (\ref{rlos}), up to $\mu_\infty$.
\end{itemize}

Summarizing the above, when $F=\mathbb Q$ we gave an alternative proof of a part of the Stark conjecture
by studying the ratios $[$ $\Gamma$-function $:$ $p$-adic $\Gamma$-function $]$, 
$[$ periods of Fermat curves $:$ $p$-adic periods of Fermat curves  $]$.
The ratio $[\exp(X(c,\iota)):\exp_p(X_p(c,\iota))]$ is a generalization of the former one.
The main result in this paper, which expresses Stark units and Gross-Stark units in terms of $\frac{\exp(X(c,\iota))}{\exp_p(X_p(c,\iota))}$'s,  
is a significant step toward the generalization of the above results in \cite{Ka2}.

\begin{rmk}
\begin{enumerate}
\item For a generalization of the ratios $[$ periods of Fermat curves $:$ $p$-adic periods of Fermat curves  $]$, 
there are results by de Shalit \cite{dS} on the ratios $[$ CM-periods $:$ $p$-adic CM periods $]$ under a certain condition.
\item Yoshida's conjecture \cite[Chap.\ III, Conjecture 3.9]{Yo} is a generalization of Rohrlich's formula.
There is a slight generalization \cite[Conjecture 5.5]{Ka3}.
\item Yoshida and the author formulated conjectures in \cite{KY1,KY2} which are generalizations of Coleman' formula.
\end{enumerate}
\end{rmk}

\begin{rmk}
Dasgupta formulated a conjecture \cite[Conjecture 3.21]{Da} which is also a refinement of Conjecture \ref{GSc}:
A modified version of $u'$ of (\ref{GSu}) is expressed in terms of a $p$-adic measure associated with zeta values.
This conjecture is a further refinement of Conjecture \ref{cnjky1} by $\ker \log_p$.
More precisely, we can show the following:
Let $u_T(\mathfrak b ,\mathcal D)$ be as in \cite[Conjecture 3.21]{Da}, 
$Y_p(\tau,\mathrm{id})$ for $\tau \in \mathrm{Gal}(H/F)$ as in Proposition \ref{eqofcnjs}-(ii),
and $(\frac{H/F}{*})$ the Artin symbol.
For simplicity we take $T:=\{\eta\}$ with $\eta$ a prime ideal of $F$. Then we can show that (\cite{Ka4})
\begin{align*}
\log_{p}(u_{\eta}(\mathfrak b ,\mathcal D))
=-Y_{p}((\tfrac{H/F}{\mathfrak b}),\mathrm{id})
+\mathrm{N}\eta\,Y_{p}((\tfrac{H/F}{\mathfrak b \eta^{-1}}),\mathrm{id}). 
\end{align*}
Here we note that $\mathfrak f$ in \cite{Da} is $\mathfrak f\mathfrak p_{\mathrm{id}}$ in this paper.
\end{rmk}

The outline of this paper is structured as follows.
In \S 2, we introduce some basic properties of Barnes' multiple zeta functions, multiple gamma functions, and their $p$-adic analogues.
In \S 3, we recall a Proposition by Shintani, which is needed in order to relate the derivative values of the classical (resp.\ $p$-adic) partial zeta functions 
to the classical (resp.\ $p$-adic) multiple gamma functions in Theorem \ref{SYf} (resp.\ Theorem \ref{Kf}) in \S 4.
We note that the classical multiple gamma functions and the $p$-adic analogues have not yet been mixed at this point. 
In \S 4, we recall the definition and some properties of Yoshida's class invariant $X(c,\iota)$.
We also provide their $p$-adic analogues by quite similar arguments.
There are two applications of these properties:
First, in \S 5, we prove the algebraicity of some products of $\exp_p(X_p(c,\iota))$'s,
which is the $p$-adic analogue of the main results in \cite{Ka3}.
Then we can clarify the relation between Stark units and the ratios $[\exp(X(c,\iota)):\exp_p(X_p(c,\iota))]$.
Next, in \S 6, we prove that a conjecture \cite[Conjecture 5.10]{KY1} on the ratios $[\exp(X(c,\iota)):\exp_p(X_p(c,\iota))]$ holds true, 
and also clarify the relation between Gross-Stark units and the same ratios. 

\section{Multiple zeta and gamma functions}

In this section, we review some basic properties of Barnes' multiple zeta and gamma functions, and their $p$-adic analogues.
For omitted proofs and details, we refer to \cite[Chap.\ I, \S 1]{Yo}, \cite{Ka1}. 
We denote by $\prn$ the set of all positive real numbers, by $\nni$ the set of all non-negative integers.

\begin{dfn} \label{mzf} 
\begin{enumerate}
\item Let $z \in \prn$, $\bm v \in \prn^r$. Then Barnes' multiple zeta function is defined as
\begin{align*}
\zeta(s,\bm v,z):=\sum_{\bm m \in \nni^r}\left(z+\bm m \ttt \bm v\right)^{-s} \quad (\mathrm{Re}(s)>r),
\end{align*}
where $\bm m \ttt \bm v$ denotes the inner product.
\item Let $k$ be a field of characteristic $0$. 
(We may assume that $k=\mathbb C$ or $\mathbb C_p$.)
Let $\bm v=(v_1,\dots,v_r)\in (k^\times)^r$, $\bm x=(x_1,\dots,x_r) \in k^r$.
Then we define ``formal multiple zeta values'' for $m \in \nni$ as
\begin{align*}
\zeta_{\mathrm{fml}}(-m,\bm v,\bm x \ttt \bm v):=(-1)^r m!\sum_{|\bm l|=m+r} \prod_{i=1}^r \frac{B_{l_i}(x_i)v_i^{l_i-1}}{l_i!} \in k.
\end{align*}
Here we denote the $l$th Bernoulli polynomial by $B_l(x)$ and $\bm l$ in the sum 
runs over all $\bm l =(l_1,\dots,l_r) \in \nni^r$ satisfying $|\bm l|:=l_1+\dots + l_r=m+r$.
\end{enumerate}
\end{dfn}

\begin{prp}
The series in the definition of $\zeta(s,\bm v,z)$ converges for $\mathrm{Re}(s)>r$, 
has a meromorphic continuation to $\mathbb C$, which is analytic at $s=0,-1,-2,\dots$.
Moreover when $\bm v \in \prn^r$, $\bm x \ttt \bm v \in \prn$, $m \in \nni$ we have 
\begin{align*}
\zeta(-m,\bm v,\bm x \ttt \bm v)=\zeta_{\mathrm{fml}}(-m,\bm v,\bm x \ttt \bm v).
\end{align*}
\end{prp}

\begin{dfn} \label{mgf}
Let $z \in \prn$, $\bm v \in \prn^r$. 
Then we define Barnes' multiple gamma function by
\begin{align*}
\Gamma(z,\bm v):=\exp\left(\frac{\partial}{\partial s}\zeta(s,\bm v,z)|_{s=0}\right).
\end{align*}
Note that this definition is modified 
from the original one $\frac{\Gamma(z,\bm v)}{\rho(z)}=\exp\left(\frac{\partial}{\partial s}\zeta(s,\bm v,z)|_{s=0}\right)$ with a correction term $\rho(z)$.
\end{dfn}

We fix embeddings $\overline{\mathbb Q} \hookrightarrow \mathbb C$, $\overline{\mathbb Q} \hookrightarrow\mathbb C_p$ throughout this paper.
In particular $p^r \in \mathbb C_p$ is well-defined for each $r \in \mathbb Q$.

\begin{dfn} \label{pmzf}
\begin{enumerate}
\item We put
\begin{align*}
P&:=\{p^r \in \mathbb C_p \mid r \in \mathbb Q\}, \\
1+M&:=\{z \in \mathbb C_p \mid |z-1|_p<1\}, \\
\mu_{(p)}&:=\{z \in \mathbb C_p \mid z^n=1 \text{ for some } n \in \mathbb N \text{ with } p \nmid n \}
\end{align*}
and consider the following  decomposition 
\begin{align*}
\mathbb C_p^\times \ &\st{\cong}\ra \quad P \quad \times \quad \mu_{(p)}\quad  \times \quad 1+M, \\
z\quad &\mt (p^{\mathrm{ord}_p(z)},\quad \theta_p(z)\quad \, ,\qquad \langle z \rangle).
\end{align*}
More precisely, let
\begin{align*}
\mathrm{ord}_p \colon \mathbb C_p^\times \ra \mathbb Q, \quad \theta_p\colon \mathbb C_p^\times \ra \mu_{(p)}
\end{align*}
be unique group homomorphisms satisfying 
\begin{align*}
|p^{-\mathrm{ord}_p(z)}\theta_p(z)^{-1}z-1|_p<1.
\end{align*}
Then we put 
\begin{align*}
\langle z \rangle:=p^{-\mathrm{ord}_p(z)}\theta_p(z)^{-1}z.
\end{align*}
\item For $z \in 1+M$, $s \in \mathbb Z_p$, we put
\begin{align*}
z^s:=\sum_{k=1}^\infty \binom{s}{k} (z-1)^k,
\end{align*}
where $\binom{s}{k}$ denotes the binomial coefficient.
\item We denote the Iwasawa $p$-adic logarithmic function by $\log_p$, which is defined as
\begin{align*}
\log_p z:=\sum_{k=1}^\infty \frac{(-1)^{k-1}(\langle z \rangle -1)^k}{k} \quad (z \in \mathbb C_p^\times).
\end{align*}
\item Let $r \in \mathbb N$. For a function $f \colon \mathbb Z_p^r \ra \mathbb C_p$, we put
\begin{align*}
I(f):=\lim_{l_1\ra \infty}\dots \lim_{l_r\ra \infty} \frac{1}{p^{l_1+\dots+l_r}} \sum_{n_1=0}^{p^{l_1}-1} \dots \sum_{n_r=0}^{p^{l_r}-1} f(n_1,\dots,n_r)
\end{align*}
whenever the limit exists.
\item Let $z \in \mathbb C_p^\times$, $\bm v=(v_1,\dots,v_r) \in \left(\mathbb C_p^\times\right)^r$. For simplicity we assume that 
\begin{align} \label{cond}
\mathrm{ord}_p(z)< \mathrm{ord}_p(v_1),\dots,\mathrm{ord}_p(v_r).
\end{align}
Then we put 
\begin{align*}
f_{z,\bm v,s} \colon \mathbb Z_p^r \ra \mathbb C_p, \ \bm x \mt \frac{(z+\bm x\ttt \bm v)^r}{v_1\dotsm v_r}\langle z+\bm x\ttt \bm v \rangle^{-s}
\end{align*}
and define the $p$-adic multiple zeta function by
\begin{align*}
\zeta_p(s,\bm v,z):=\frac{(-1)^rI(f_{z,\bm v,s})}{(r-s)(r-1-s)\dotsm(1-s)}.
\end{align*}
\end{enumerate}
\end{dfn}

\begin{prp} \label{prpofzp}
\begin{enumerate}
\item $\zeta_p(s,\bm v,z)$ is analytic at $s=0$, and continuous on $s \in \mathbb Z_p-\{1,2,\dots,r\}$.
\item $\zeta_p(s,\bm v,z)$ is continuous for $\bm v,z$.
\item Let $z \in \overline{\mathbb Q}$, $\bm v=(v_1,\dots,v_r) \in \overline{\mathbb Q}^r$. We assume that 
\begin{align*}
&z,v_1,\dots,v_r \in \prn\ (\text{via the embedding } \overline{\mathbb Q} \hookrightarrow \mathbb C), \\
&\mathrm{ord}_p(z)< \mathrm{ord}_p(v_1),\dots,\mathrm{ord}_p(v_r)\ (\text{via the embedding } 
\overline{\mathbb Q} \hookrightarrow \mathbb C_p).
\end{align*}
Then $\zeta_p(s,\bm v,z)$ satisfies the following $p$-adic interpolation property:
\begin{align*}
\zeta_p(-m,\bm v,z)=p^{-\mathrm{ord}_p(z) m}\theta_p(z)^{-m} \zeta(-m,\bm v,z) \quad (m \in \nni).
\end{align*}
\end{enumerate}
\end{prp}

\begin{proof}
(i), (iii) follow from \cite[Lemma 5.3-1, Theorem 5.1]{Ka1} respectively.
Let $z,z' \in \mathbb C_p^\times$, $\bm v,\bm v' \in \left(\mathbb C_p^\times\right)^r$ satisfy (\ref{cond}).
When $z' \ra z$, we may assume that $\mathrm{ord}_p(z)=\mathrm{ord}_p(z')$, $\theta_p(z)=\theta_p(z')$. 
Then we see that for $m \in \nni$
\begin{align*}
\sup_{\bm x \in \mathbb Z_p^r}|f_{z,\bm v,-m}(\bm x)-f_{z',\bm v',-m}(\bm x)|_p \ra 0 \ (z' \ra z,\ \bm v' \ra \bm v).
\end{align*}
Hence by \cite[(5.3)]{Ka1} we have 
\begin{align*}
|\zeta_p(-m,\bm v,z)-\zeta_p(-m,\bm v',z')|_p \ra 0 \ (z' \ra z, \bm v' \ra \bm v).
\end{align*}
Since non-positive integers are dense in $\mathbb Z_p$, the assertion (ii) is clear.
\end{proof}

\begin{dfn} \label{pmgf}
Let $z \in \mathbb C_p^\times$, $\bm v \in \left(\mathbb C_p^\times\right)^r$ satisfy (\ref{cond}).
Then we define the $p$-adic $\log$ multiple gamma function by
\begin{align*}
L\Gamma_p(z,\bm v):=\frac{\partial}{\partial s}\zeta_p(s,\bm v,z)|_{s=0}.
\end{align*}
\end{dfn}

\begin{prp} \label{prpofmg}
\begin{enumerate}
\item Let $z \in \prn$, $\bm v \in \prn^r$, $\alpha \in \prn$. Then we have
\begin{align*}
\zeta(s,\alpha \bm v,\alpha z)&=\alpha^{-s}\zeta(s,\bm v,z), \\
\Gamma(\alpha z,\alpha \bm v)&=\Gamma(z,\bm v)\alpha^{-\zeta(0,\bm v,z)}.
\end{align*}
\item Let $z \in \mathbb C_p^\times$, $\bm v \in \left(\mathbb C_p^\times\right)^r$ satisfy (\ref{cond}).
Let $\alpha \in \mathbb C_p^\times$.
Then we have
\begin{align*}
\zeta_p(s,\alpha \bm v,\alpha z)&=\langle\alpha\rangle^{-s}\zeta_p(s,\bm v,z), \\
L\Gamma_p(\alpha z,\alpha \bm v)&=L\Gamma_p(z,\bm v)-\zeta_p(0,\bm v,z)\log_p \alpha.
\end{align*}
\end{enumerate}
\end{prp}

\begin{proof}
Follows immediately from Definitions \ref{mzf}-(i), \ref{mgf}, \ref{pmzf}, \ref{pmgf}. 
\end{proof}

\section{Shintani domains}

\begin{dfn}
Let $F$ be a totally real field of degree $n$. 
\begin{enumerate}
\item Let $\bm v_1,\dots,\bm v_r \in \mathbb R^n$ be linearly independent vectors.
We define an ($r$-dimensional open simplicial) cone with basis $\bm v_1,\dots,\bm v_r$ as 
\begin{align*}
C(\bm v_1,\dots,\bm v_r):=\{t_1\bm v_1+\dots+t_r\bm v_r \mid t_i \in \prn\} \subset \mathbb R^n.
\end{align*}
\item We identify 
\begin{align*}
F\otimes \mathbb R = \mathbb R^n, \quad \sum_{i=1}^k \alpha_i\otimes \beta_i \mt \left(\sum_{i=1}^k \iota(\alpha_i)\beta_i \right)_{\iota \in \HFR}.
\end{align*}
In particular, we regard $F$ as a subset of $F\otimes \mathbb R=\mathbb R^n$.
We say that a cone $C(\bm v_1,\dots,\bm v_r)$ is a cone of $F$ if $\bm v_1,\dots,\bm v_r \in \mathcal O_F$.
\item Considering $F \subset F\otimes \mathbb R = \mathbb R^n$, we put
\begin{align*}
F\otimes \mathbb R_+ &:=\prn^n, \\
F_+&:=F \cap (F\otimes \mathbb R_+)=\{z \in F \mid \iota \in \HFR \R \iota(z)>0\}, \\
\mathcal O_{F,+}&:=\mathcal O_F \cap F_+, \\
E_{F,+}&:=\mathcal O_F^\times \cap F_+.
\end{align*}
\end{enumerate}
\end{dfn}

\begin{thm}[{\cite[Proposition 4]{Shin1}}] \label{Sd}
We consider the natural action $E_{F,+} \curvearrowright F\otimes \mathbb R_+$, $u(\alpha\otimes\beta):=(u\alpha)\otimes\beta$.
Then there exists a fundamental domain $D$ of $F\otimes \mathbb R_+/E_{F,+}$ which can be written as a finite disjoint union of cones of $F$.
Namely there exist $v_{ij} \in \mathcal O_{F,+}$ $(j \in J,1\leq i\leq r(j),\ |J|<\infty,\ r(j) \in \mathbb N)$ satisfying  
\begin{align*}
F\otimes \mathbb R_+=\coprod_{u \in E_{F,+}} uD, \quad 
D=\coprod_{j \in J} C(v_{j1},\dots,v_{jr(j)}).
\end{align*}
We call such a $D$ a Shintani domain.
\end{thm}

\section{The class invariants $X(c,\iota)$, $X_p(c,\iota)$}

In this section, we summarize the definitions and some properties of $X(c,\iota)$ and $X_p(c,\iota)$.
The class invariant $X(c)$ was introduced and studied by Yoshida \cite{Yo} ($X(c,\iota)$ in this paper is almost equal to $X(\iota(c))$ in \cite{Yo}).
Yoshida and the author studied the $p$-adic analogue $X_p(c)$ in \cite{KY1,KY2}.
Throughout this section let $F$ be a totally real field of degree $n$, $\mathfrak f$ an integral ideal of $F$, 
$c \in C_\mathfrak f$, $\iota \in \HFR$, and $D=\coprod_{j \in J}C(\bm v_j)$ a disjoint union of cones with $\bm v_j \in \mathcal O_F^{r(j)}$.
Unless otherwise specified, we do not assume that $D$ is a Shintani domain in the sense of Theorem \ref{Sd}.
Let $\pi \colon C_\mathfrak f \ra C_{(1)}$ be the natural projection.
We take an integral ideal $\mathfrak a_c$ satisfying $\mathfrak a_c \mathfrak f \in \pi(c)$ for each $c \in C_\mathfrak f$.

\subsection{The case of a totally positive $D$}

In this subsection, we assume that $D=\coprod_{j \in J}C(\bm v_j)$ ($\bm v_j \in \mathcal O_F^{r(j)}$) is totally positive, that is
\begin{align*}
D \subset F\otimes \mathbb R_+, \text{ or equivalently, } \bm v_j \in \mathcal O_{F,+}^{r(j)}.
\end{align*}
We introduce Yoshida's invariants $G,W,V,X$ in \cite[Chap.\ III, (3.6)--(3.9), (3.28)--(3.31)]{Yo}, which is slightly modified in \cite[(4.3)]{KY1}, \cite[\S 2]{Ka3}.
The equalities in Definitions \ref{G}, \ref{W} follow from \cite[Chap.\ II, Lemma 3.2]{Yo}.

\begin{dfn} \label{R}
Let $\bm v=(v_1,\dots,v_r) \in \mathcal O_{F,+}^r$ be linearly independent. Then we put
\begin{align*}
R(c,\bm v):=R(c,\bm v,\mathfrak a_c):=\left\{\bm x \in ((0,1] \cap \mathbb Q)^r \mid 
\mathcal O_F \supset (\bm x\ttt \bm v)\mathfrak a_c\mathfrak f \in c\right\}.
\end{align*}
\end{dfn}

\begin{dfn} \label{G}
We put 
\begin{align*}
G(c,\iota,D,\mathfrak a_c)
:=\left[\frac{d}{ds}\sum_{z \in ({\mathfrak a}_c \mathfrak f)^{-1} \cap D,\,(z){\mathfrak a}_c \mathfrak f \in c} \iota(z)^{-s}\right]_{s=0}
=\sum_{j \in J}\sum_{\bm x \in R(c,\bm v_j)}\log \Gamma(\iota(\bm x\ttt\bm v_j),\iota(\bm v_j)).
\end{align*}
\end{dfn}

\begin{dfn} \label{logiota}
We define group homomorphisms $\log_\iota \colon I_F \ra \mathbb R$ for $\iota \in \HFR$ as follows:
For each prime ideal $\mathfrak p$ of $F$, we choose a generator $\pi_\mathfrak p \in \mathcal O_{F,+}$ 
of the principal ideal $\mathfrak p^{h_{F,+}}$, where $h_{F,+}=|C_{(1)}|$ is the narrow class number.
Then we put
\begin{align*}
\log_\iota \mathfrak p:=\frac{1}{h_{F,+}} \log \iota(\pi_\mathfrak p).
\end{align*}
We extend this linearly to $\log_\iota \colon I_F \ra \mathbb R$.
\end{dfn}

We consider a principal ideal $(\alpha)=\alpha\mathcal O_F$ with $\alpha \in F^\times$.
The difference between $\log \iota(\alpha)$ and $\log_{\iota} (\alpha)$ is as follows:
By definition there exists a generator $\alpha' \in (\alpha)^{h_{F,+}}$ satisfying $\log_{\iota} (\alpha)=\frac{1}{h_{F,+}}\log \iota(\alpha')$.
Then we see that 
\begin{align} \label{ea}
\log \iota(\alpha)-\log_{\iota} (\alpha)=\frac{1}{h_{F,+}}\log \iota(u_\alpha)
\quad  (u_\alpha:=\frac{\alpha^{h_{F,+}}}{\alpha'} \in E_{F,+}).
\end{align}

\begin{dfn} \label{W}
We put 
\begin{align*}
W(c,\iota,D,\mathfrak a_c)
&:=-\left[\sum_{z \in ({\mathfrak a}_c \mathfrak f)^{-1} \cap D,\,(z){\mathfrak a}_c \mathfrak f \in c} \iota(z)^{-s}\right]_{s=0} 
\log_\iota \mathfrak a_c\mathfrak f \\
&=-\left(\sum_{j \in J}\sum_{\bm x \in R(c,\bm v_j)}\zeta(0,\iota(\bm v_j),\iota(\bm x\ttt\bm v_j))\right) \log_\iota \mathfrak a_c\mathfrak f.
\end{align*}
\end{dfn}

\begin{dfn} \label{V}
For $\iota,\iota' \in \HFR$ with $\iota\neq \iota'$, we put 
\begin{align*}
v_{\iota,\iota'}:=\left[\frac{d}{ds}\sum_{z \in ({\mathfrak a}_c \mathfrak f)^{-1} \cap D,\,(z){\mathfrak a}_c \mathfrak f \in c} 
\left((\iota(z)\iota'(z))^{-s} -\iota(z)^{-s}-\iota'(z)^{-s} \right)\right]_{s=0}.
\end{align*}
Then we define 
\begin{align*}
V(c,\iota,D,\mathfrak a_c)&:=\frac{2}{n}\sum_{\iota' \neq \iota} v_{\iota,\iota'} -\frac{2}{n^2}\sum_{\iota'\neq \iota''} v_{\iota',\iota''}.
\end{align*}
Here $\iota'$ runs over all $\iota' \in \HFR$ with $\iota'\neq \iota$ in the first sum,
$\iota',\iota''$ run over all $\iota',\iota'' \in \HFR$ with $\iota'\neq \iota''$ in the second sum.
\end{dfn}

\begin{dfn} \label{X}
We put 
\begin{align*}
X(c,\iota,D,\mathfrak a_c):=G(c,\iota,D,\mathfrak a_c)+W(c,\iota,D,\mathfrak a_c)+V(c,\iota,D,\mathfrak a_c).
\end{align*}
If $D$ is a Shintani domain in the sense of Theorem \ref{Sd}, and if we fix $D,\mathfrak a_c$, then we put 
\begin{align*}
X(c,\iota):=X(c,\iota,D,\mathfrak a_c),\ G(c,\iota):=G(c,\iota,D,\mathfrak a_c).
\end{align*}
\end{dfn}

The following Theorem is essentially due to Yoshida:
Yoshida transformed Shintani's formula (\ref{Sf}) in \cite{Shin2} into a similar form.
Yoshida and the author modified the $W$-term slightly in \cite{KY1}, \cite{Ka3} into the present form.

\begin{thm}[{\cite[Chap.\ III, (3.11)]{Yo}, \cite[Theorem 2.5]{Ka3}}] \label{SYf}
Let $D$ be a Shintani domain. Then we have
\begin{align*}
\zeta'(0,c)=\sum_{\iota \in \HFR}X(c,\iota).
\end{align*}
\end{thm}

The following Lemma also is essentially due to Yoshida, 
although $\exp(X(\iota(c)))$ in \cite{Yo} was well-defined up to $\iota(F_+)^\mathbb Q$.

\begin{lmm}[{\cite[Chap.\ III, \S 3.6, 3.7]{Yo}, \cite[Lemma 3.11]{Ka3}}] \label{123}
We consider the following operations on $D,\mathfrak a_c$.
\begin{enumerate}
\item[\I] Let $j_0 \in J$. We decompose $C(\bm v_{j_0})=\coprod_{l \in L}C(\bm v_l)$ and replace
\begin{align*}
D=\coprod_{j \in J}C(\bm v_j) \R D=\left(\coprod_{j \in J-\{j_0\}}C(\bm v_j)\right) \coprod \left(\coprod_{l \in L}C(\bm v_l)\right).
\end{align*}
Note that a replacement of basis $(|L|=1)$ is included in this case.
\item[\II] Let $j_0 \in J$, $\epsilon \in E_{F,+}$. We replace $C(\bm v_{j_0})$ by $C(\epsilon \bm v_{j_0})$, that is
\begin{align*}
D=\coprod_{j \in J}C(\bm v_j) \R D=\left(\coprod_{j \in J-\{j_0\}}C(\bm v_j)\right) \coprod C(\epsilon \bm v_{j_0}).
\end{align*}
\item[\III] Let $\alpha \in \mathfrak a_{c}^{-1} \cap F_+$.  We replace 
\begin{align*}
\mathfrak a_c &\R \alpha \mathfrak a_c, \\
C(\bm v_j) &\R \alpha^{-1}C(\bm v_j)=C(\bm v_j')
\end{align*}
simultaneously. 
Here we take $\bm v_j'$ satisfying $\alpha^{-1}C(\bm v_j)=C(\bm v_j')$, $\bm v_j' \in \mathcal O_{F,+}^{r(j)}$ for each $j \in J$.
\end{enumerate}
In order to simplify the following expressions, we put
\begin{align*}
Z_j :=\sum_{\bm x \in R(c,\bm v_j)}\zeta_{\mathrm{fml}}(0,\bm v_j,\bm x \ttt \bm v_j) \in F \quad (j \in J).
\end{align*}
In particular
\begin{align*}
\iota(Z_j)&=\sum_{\bm x \in R(c,\bm v_j)}\zeta(0,\iota(\bm v_j),\iota(\bm x \ttt \bm v_j)) \in \iota(F), \\
\mathrm{Tr}_{F/\mathbb Q}Z_j&=\sum_{\bm x \in R(c,\bm v_j)}\sum_{\iota \in \HFR}\zeta(0,\iota(\bm v_j),\iota(\bm x \ttt \bm v_j)) \in \mathbb Q
\end{align*}
have meanings.
\begin{enumerate}
\item $G(c,\iota,D, \mathfrak a_c)$ changes as follows by the operations \I, \II, \III.
\begin{enumerate}
\item[\I] Stays constant.
\item[\II] $\displaystyle G(c,\iota,D,\mathfrak a_c)-\iota(Z_{j_0})\log\iota(\epsilon)$.
\item[\III] $\displaystyle G(c,\iota,D,\mathfrak a_c)+\sum_{j\in J}\iota(Z_j)\log\iota(\alpha)$.
\end{enumerate}
\item $W(c,\iota,D, \mathfrak a_c)$ changes as follows by the operations \I, \II, \III.
\begin{enumerate}
\item[\I] Stays constant.
\item[\II] Stays constant.
\item[\III] $\displaystyle W(c,\iota,D,\mathfrak a_c)-\sum_{j\in J}\iota(Z_j)\log_{\iota}(\alpha)$.
\end{enumerate}
\item $V(c,\iota,D, \mathfrak a_c)$ changes as follows by the operations \I, \II, \III.
\begin{enumerate}
\item[\I] Stays constant.
\item[\II] $\displaystyle V(c,\iota,D,\mathfrak a_c) + \left(\iota(Z_{j_0})-\frac{\mathrm{Tr}_{F/\mathbb Q}(Z_{j_0})}{n}\right)\log\iota(\epsilon)$.
\item[\III] $\displaystyle V(c,\iota,D,\mathfrak a_c) +\sum_{j \in J} \left(\iota(Z_j)-\frac{\mathrm{Tr}_{F/\mathbb Q}(Z_j)}{n}\right)
\left(\frac{1}{n}\log (N_{F/\mathbb Q}(\alpha)) -\log\iota(\alpha)\right)$.
\end{enumerate}
\item Additionally assume that $D$ is a Shintani domain. Then $X(c,\iota,D, \mathfrak a_c)$ changes as follows by the operations \I, \II, \III.
\begin{enumerate}
\item[\I] Stays constant.
\item[\II] $\displaystyle X(c,\iota,D,\mathfrak a_c) -\frac{\mathrm{Tr}_{F/\mathbb Q}(Z_j)}{n}\log\iota(\epsilon)$.
\item[\III] $\displaystyle X(c,\iota,D,\mathfrak a_c)-\frac{\zeta(0,c)}{h_{F,+}}\log \iota(u_\alpha)$.
Here we take $u_\alpha \in E_{F,+}$ is as in (\ref{ea}).
\end{enumerate}
In particular, $\exp(X(c,\iota)) \bmod \iota(E_{F,+})^\mathbb Q$ does not depend on $D$, $\mathfrak a_c$, 
where we put $\iota(E_{F,+})^\mathbb Q:=\{\iota(u)^{\frac{1}{N}} \mid u \in E_{F,+},\ N \in \mathbb N\}$.
More precisely, for Shintani domains $D,D'$, 
integral ideals $\mathfrak a_c,\mathfrak a_c'$ satisfying $\mathfrak a_c\mathfrak f,\mathfrak a_c'\mathfrak f \in \pi(c)$, 
there exist $N \in \mathbb N$, $u_{c,D,D',\mathfrak a_c,\mathfrak a_c'} \in E_{F,+}$ satisfying 
\begin{align*}
X(c,\iota,D',\mathfrak a_c')-X(c,\iota,D,\mathfrak a_c)=\frac{1}{N}\log \iota(u_{c,D,D',\mathfrak a_c,\mathfrak a_c'}).
\end{align*}
\end{enumerate}
\end{lmm}

\begin{proof}
For comparison to the $p$-adic case (Lemma \ref{123p}), we give a brief sketch of the poof of (i)-\II.
Let $D':=(\coprod_{j \in J-\{j_0\}}C(\bm v_j)) \coprod C(\epsilon \bm v_{j_0})$.
By definition we have
\begin{align*}
&G(c,\iota,D',\mathfrak a_c)-G(c,\iota,D,\mathfrak a_c) \\
&=\left[\frac{d}{ds}\sum_{z \in ({\mathfrak a}_c \mathfrak f)^{-1} \cap C(\epsilon \bm v_{j_0}),\,(z){\mathfrak a}_c \mathfrak f \in c} \iota(z)^{-s}\right]_{s=0}
-\left[\frac{d}{ds}\sum_{z \in ({\mathfrak a}_c \mathfrak f)^{-1} \cap C(\bm v_{j_0}),\,(z){\mathfrak a}_c \mathfrak f \in c} \iota(z)^{-s}\right]_{s=0}.
\end{align*}
By noting that  
\begin{align} \label{eqofsets}
\{z \in ({\mathfrak a}_c \mathfrak f)^{-1} \cap C(\epsilon \bm v_{j_0}) \mid (z){\mathfrak a}_c \mathfrak f \in c\}
=\epsilon \{z \in ({\mathfrak a}_c \mathfrak f)^{-1} \cap C(\bm v_{j_0}) \mid (z){\mathfrak a}_c \mathfrak f \in c\},
\end{align}
we can rewrite 
\begin{align*}
G(c,\iota,D',\mathfrak a_c)-G(c,\iota,D,\mathfrak a_c) 
&=\left[\frac{d}{ds}\left(\iota(\epsilon)^{-s}-1\right)
\sum_{z \in ({\mathfrak a}_c \mathfrak f)^{-1} \cap C(\bm v_{j_0}),\,(z){\mathfrak a}_c \mathfrak f \in c} \iota(z)^{-s}\right]_{s=0} \\
&=-\iota(Z_{j_0})\log\iota(\epsilon)
\end{align*}
as desired.
The other cases are similar.
\end{proof}

\subsection{The case of a non-totally positive $D$}

In this subsection, we fix an embedding $\iota_0 \in \HFR$.
We put 
\begin{align*}
F\otimes \mathbb R_{(n-1)+} :=\left\{\sum_{i=1}^k\alpha_i \otimes \beta_i \in F\otimes \mathbb R \mid 
\HFR \ni \iota \neq \iota_0 \R \sum_{i=1}^k\iota(\alpha_i) \beta_i \in \prn \right \}.
\end{align*}
For each subset $A \subset F$, we put
\begin{align*}
A_{(n-1)+}:=A \cap \left(F\otimes \mathbb R_{(n-1)+}\right)=\{a \in A \mid \HFR \ni  \iota \neq \iota_0 \R \iota(a) \in \prn\}.
\end{align*}
We introduce generalizations of Yoshida's invariants in \cite{Ka3}.

\begin{dfn} \label{ciota}
For each $\iota \in \HFR$, we take $\nu_\iota \in \mathcal O_F$ satisfying 
\begin{align*}
\nu_\iota \equiv 1 \bmod \mathfrak f,\ \iota(\mu_\iota)<0,\ \iota'(\nu_\iota)>0\ (\iota\neq \iota' \in \HFR)
\end{align*}
and put
\begin{align*}
c_{\iota} :=[(\nu_\iota)] \in C_\mathfrak f.
\end{align*}
\end{dfn}

\begin{rmk} \label{sccase}
Let $H_\mathfrak f$ be the maximal ray class field modulo $\mathfrak f$ in the narrow sense,
$\mathrm{Art}\colon C_\mathfrak f \ra \mathrm{Gal}(H_\mathfrak f/F)$ the Artin map.
Then $\mathrm{Art}(c_{\iota})$ is the complex conjugation at $\iota$ \cite[Chap.\ III, the first paragraph of \S 5.1]{Yo}.
Hence the fixed subfield $H_{\mathfrak f}^{\langle \mathrm{Art}(c_\iota) \rangle}$ is the maximal subfield where the real place $\iota$ splits completely.
\end{rmk}

\begin{dfn}
Let $\bm v=(v_1,\dots,v_r) \in \mathcal O_{F,(n-1)+}^r$ be linearly independent. 
Then we put
\begin{align*}
R(c \cup cc_{\iota_0},\bm v):=R(c\cup cc_{\iota_0},\bm v,\mathfrak a_c):=\left\{\bm x \in ((0,1] \cap \mathbb Q)^r \mid 
\mathcal O_F \supset (\bm x\ttt \bm v)\mathfrak a_c\mathfrak f \in c \cup cc_{\iota_0}\right\}.
\end{align*}
\end{dfn}

\begin{dfn} \label{G2}
Let $\iota \in \HFR$, $\neq \iota_0$.
By replacing $c$ with $c \cup cc_{\iota_0}$ in Definition \ref{G}, we put 
\begin{align*}
G(c \cup cc_{\iota_0},\iota,D,\mathfrak a_c)
&:=\left[\frac{d}{ds}\sum_{z \in ({\mathfrak a}_c \mathfrak f)^{-1} \cap D,\,(z){\mathfrak a}_c \mathfrak f \in c \cup cc_{\iota_0}} \iota(z)^{-s}\right]_{s=0} \\
&=\sum_{j \in J}\sum_{\bm x \in R(c \cup cc_{\iota_0},\bm v_j)}\log \Gamma(\iota(\bm x\ttt\bm v_j),\iota(\bm v_j)).
\end{align*}
\end{dfn}

\begin{rmk}
When $C(\bm v) \subset F\otimes \mathbb R_{(n-1)+}$, we consider the pair $c,cc_{\iota_0}$
for the following reason:
Let $z \in ({\mathfrak a}_c \mathfrak f)^{-1} \cap C(\bm v)$ satisfy $(z){\mathfrak a}_c \mathfrak f \in c$.
We write $\bm v=(v_1,\dots,v_r)$ and $z=\bm x \ttt\bm v$ with $\bm x \in \mathbb Q_{>0}^r$.
If $C(\bm v) \subset F\otimes \mathbb R_+$, then we see that 
\begin{align*}
(z+ v_i){\mathfrak a}_c \mathfrak f &\in c, \\
(z-v_i){\mathfrak a}_c \mathfrak f &\in c\ (x_i>1),
\end{align*}
by noting that $(z\pm v_i){\mathfrak a}_c \mathfrak f=(1\pm v_i/z)(z){\mathfrak a}_c \mathfrak f$. 
It follows that $\sum_{z \in ({\mathfrak a}_c \mathfrak f)^{-1} \cap C(\bm v),\,(z){\mathfrak a}_c \mathfrak f \in c} \iota(z)^{-s}
=\sum_{\bm x \in R(c,\bm v)}\zeta(s,\iota(\bm v),\iota(\bm x\ttt\bm v))$.
If $C(\bm v) \subset F\otimes \mathbb R_{(n-1)+}$, then it follows only that $(z\pm v_i){\mathfrak a}_c \mathfrak f \in c\cup cc_{\iota_0}$.
\end{rmk}

\begin{dfn} 
For each $\iota \in \HFR$ (including $\iota=\iota_0$), we put 
\begin{align*}
W(c \cup cc_{\iota_0},\iota,D,\mathfrak a_c)
:=-\left(\sum_{j \in J}\sum_{\bm x \in R(c \cup cc_{\iota_0},\bm v_j)}\zeta_{\mathrm{fml}}(0,\iota(\bm v_j),\iota(\bm x\ttt\bm v_j))\right) 
\log_\iota \mathfrak a_c\mathfrak f.
\end{align*}
We note that $\zeta_{\mathrm{fml}}(0,\iota(\bm v_j),\iota(\bm x\ttt\bm v_j))=\zeta(0,\iota(\bm v_j),\iota(\bm x\ttt\bm v_j))$
if all entries of $\iota(\bm v_j)$ are positive.
\end{dfn}

\begin{dfn} \label{V2}
\begin{enumerate}
\item Let $\bm v=(v_1,\dots,v_r) \in \mathcal O_F^r$ be linearly independent.
Let $\bm x=(x_1,\dots,x_r) \in (\mathbb Q \cap (0,1])^r$, $\iota \in \HFR$. We put
\begin{align*}
V(\bm v,\bm x,\iota)&:=
\frac{-1}{n}\sum_{\substack{\iota' \in \HFR \\ \iota'\neq \iota}} \sum_{p=1}^r
\zeta_{\mathrm{fml}}(-1,(\tfrac{\iota(v_q)}{\iota(v_p)}-\tfrac{\iota'(v_q)}{\iota'(v_p)})_{q\neq p},(x_q)_{q\neq p})
\log|\tfrac{\iota(v_p)}{\iota'(v_p)}| \\
&\qquad +\frac{1}{2n^2}\sum_{\substack{\iota',\iota'' \in \HFR \\ \iota'\neq \iota''}}
\sum_{p=1}^r \zeta_{\mathrm{fml}}(-1,(\tfrac{\iota'(v_q)}{\iota'(v_p)}-\tfrac{\iota''(v_q)}{\iota''(v_p)})_{q\neq p},(x_q)_{q\neq p})\log|\tfrac{\iota'(v_p)}{\iota''(v_p)}|.
\end{align*}
Here for an $r$-dimensional vector $(a_q)=(a_1,\dots,a_r)$, we denote by $(a_q)_{q\neq p}$ the $(r-1)$-dimensional vector $(a_1,\dots,a_{p-1},a_{p+1},\dots,a_r)$.
\item For each $\iota \in \HFR$ (including $\iota=\iota_0$), we put 
\begin{align*}
V(c \cup cc_{\iota_0},\iota,D,\mathfrak a_c):=\sum_{j \in J} \sum_{\bm x \in R(c\cup cc_{\iota_0},\bm v_j)} V(\bm v_j,\bm x,\iota).
\end{align*}
\end{enumerate}
\end{dfn}

\begin{dfn}
Let $\iota \in \HFR$, $\neq \iota_0$. We put 
\begin{align*}
X(c \cup cc_{\iota_0},\iota,D,\mathfrak a_c)
:=G(c \cup cc_{\iota_0},\iota,D,\mathfrak a_c)+W(c \cup cc_{\iota_0},\iota,D,\mathfrak a_c)+V(c \cup cc_{\iota_0},\iota,D,\mathfrak a_c).
\end{align*}
\end{dfn}

\begin{lmm}[{\cite[Lemmas 3.6, 3.7, 3.10]{Ka3}}] \label{1232}
Let $\iota \in \HFR$, $\neq \iota_0$. 
Then the same statements as in Lemma \ref{123} hold, by replacing $c$ with $c \cup cc_{\iota_0}$.
In particular we replace $Z_j$ with 
\begin{align*}
Z_j =\sum_{\bm x \in R(c \cup cc_{\iota_0},\bm v_j)}\zeta_{\mathrm{fml}}(0,\bm v_j,\bm x \ttt \bm v_j).
\end{align*}
For {\rm (iv)-\III}, we replace $\zeta(0,c)$ with
\begin{align*}
\zeta(0,c \cup cc_{\iota_0}):=
\begin{cases}
\zeta(0,c)&(c_{\iota_0}=[(1)] \text{ or }\mathcal O_{F,(n-1)+}^{\times}=\emptyset), \\
\zeta(0,c)+\zeta(0, cc_{\iota_0})&(\text{otherwise}).
\end{cases}
\end{align*}
\end{lmm}

\begin{lmm}[{\cite[Lemma 3.12]{Ka3}}] \label{12}
We assume that $D$ is a Shintani domain, and that $\iota\neq \iota_0$.
Then the following assertions hold.
\begin{enumerate}
\item We have
\begin{align*}
X(cc_{\iota_0}\cup c,\iota,\nu_{\iota_0}^{-1} D,(\nu_{\iota_0})\mathfrak a_c)-X(c\cup cc_{\iota_0},\iota,D,\mathfrak a_c) 
= \frac{\zeta(0,c)}{h_{F,+}} \log \iota(u_{\nu_{\iota_0}}) .
\end{align*}
Here, in the symbol $X(cc_{\iota_0}\cup c,\iota,\nu_{\iota_0}^{-1} D,(\nu_{\iota_0})\mathfrak a_c)$, the roles of $c,c_{\iota_0}$ are exchanged.
We take $u_{\nu_{\iota_0}} \in E_{F,+}$ is as in (\ref{ea}).
\item If $c_{\iota_0}=[(1)]$ or if $\mathcal O_{F,(n-1)+}^{\times}=\emptyset$, then we have 
\begin{align*}
X(c\cup cc_{\iota_0},\iota, D,\mathfrak a_c)=X(c,\iota,D,\mathfrak a_c).
\end{align*}
Otherwise, we take an element $\epsilon\in \mathcal O_{F,(n-1)+}^{\times}$. Then we have  
\begin{align*}
&X(c\cup cc_{\iota_0},\iota, D,\mathfrak a_c) \\
&=X(c,\iota,D,\mathfrak a_c)+X(cc_{\iota_0},\iota,\nu_{\iota_0}^{-1}\epsilon D,(\nu_{\iota_0})\mathfrak a_c) 
+\frac{\zeta(0,cc_{\iota_0})}{h_{F,+}}\log \iota(u_{\nu_{\iota_0}\epsilon^{-1}}).
\end{align*}
Here we take $u_{\nu_{\iota_0}\epsilon^{-1}} \in E_{F,+}$ is as in (\ref{ea}).
\end{enumerate}
\end{lmm}

\subsection{$p$-adic analogues}

We introduce the $p$-adic analogues of Yoshida's invariants.
The case when $D \subset F\otimes \mathbb R_+$ was studied in \cite{KY1,KY2}.
In the $p$-adic case, the Archimedean topology induced by $\iota \in \HFR$ is not so important.
Instead, we consider the prime ideal corresponding to the $p$-adic topology.

\begin{dfn} \label{piota}
Recall that we fixed embeddings $\overline{\mathbb Q} \hookrightarrow \mathbb C$, $\overline{\mathbb Q} \hookrightarrow\mathbb C_p$.
We identify 
\begin{align*}
\HFR =\HFCp
\end{align*}
and define for $\iota \in \HFR$ $(=\HFCp)$
\begin{align*}
\mfpi:=\{z \in \mathcal O_F \mid |\iota(z)|_p<1 \}.
\end{align*}
In other words, $\iota(\mfpi)$ corresponds to the $p$-adic topology on $\iota(F) \subset \mathbb C_p$.
\end{dfn}

We fix $\iota_0 \in \HFR$ and define $F\otimes \mathbb R_{(n-1)+} \subset F\otimes \mathbb R$ as in the previous subsection. 
We note that $\iota$ may be equal to $\iota_0$. 
Throughout this subsection, we assume that
\begin{align*}
\mfpi \mid \mathfrak f.
\end{align*}
Under this assumption $(\iota(\bm v_j),\iota(\bm x\ttt\bm v_j))$ with $\bm x \in R(c,\bm v_j)$ satisfy (\ref{cond}). 
Namely, the $p$-adic interpolation function $\zeta_p(s,\iota(\bm v_j),\iota(\bm x\ttt\bm v_j))$ 
of $\zeta(s,\iota(\bm v_j),\iota(\bm x\ttt\bm v_j))$ is well-defined.
Hence we may consider the $p$-adic interpolation of 
$\sum_{z \in ({\mathfrak a}_c \mathfrak f)^{-1} \cap D,\,(z){\mathfrak a}_c \mathfrak f \in c} \iota(z)^{-s}$ in Definition \ref{G}.
In the setting of Definition \ref{G2}, 
$\zeta_p(s,\iota(\bm v_j),\iota(\bm x\ttt\bm v_j))$ for $\bm x \in R(c\cup cc_{\iota_0},\bm v_j)$ are also well-defined 
whenever $\mfpi \mid \mathfrak f$, 
although $\zeta(s,\iota(\bm v_j),\iota(\bm x\ttt\bm v_j))$ with $\iota=\iota_0$ may not be well-defined.

\begin{dfn} 
When $D \subset F\otimes \mathbb R_{(n-1)+}$, we put 
\begin{align*}
G_p(c\cup cc_{\iota_0},\iota,D,\mathfrak a_c):=\sum_{j \in J}\sum_{\bm x \in R(c\cup cc_{\iota_0},\bm v_j)}L\Gamma_p(\iota(\bm x\ttt\bm v_j),\iota(\bm v_j)).
\end{align*}
When $D \subset F\otimes \mathbb R_+$, we also define $G_p(c,\iota,D,\mathfrak a_c)$ by replacing 
$R(c\cup cc_{\iota_0},\bm v_j)$ with $R(c,\bm v_j)$.
\end{dfn}

\begin{dfn}
Let $h_{F,+},\pi_\mathfrak p$ be as in Definition \ref{logiota} and we put
\begin{align*}
\log_{\iota,p} \mathfrak p:=\frac{1}{h_{F,+}} \log_p \iota(\pi_\mathfrak p)
\end{align*}
for prime ideals $\mathfrak p$ of $F$ and for $\iota \in \HFR$.
We extend this linearly to $\log_{\iota,p} \colon I_F \ra \mathbb C_p$.
When $D \subset F\otimes \mathbb R_{(n-1)+}$, we put 
\begin{align*}
W_p(c\cup cc_{\iota_0},\iota,D,\mathfrak a_c)
:=-\left(\sum_{j \in J}\sum_{\bm x \in R(c\cup cc_{\iota_0},\bm v_j)}\zeta_{\mathrm{fml}}(0,\iota(\bm v_j),\iota(\bm x\ttt\bm v_j))\right) 
\log_{\iota,p} \mathfrak a_c\mathfrak f.
\end{align*}
When $D \subset F\otimes \mathbb R_+$, we also define $W_p(c,\iota,D,\mathfrak a_c)$ similarly.
\end{dfn}

\begin{dfn}
We define $V_p(\bm v,\bm x,\iota)$ by replacing $\log |\cdots|$ in Definition \ref{V2} with $\log_p (\cdots)$.
When $D \subset F\otimes \mathbb R_{(n-1)+}$, we put 
\begin{align*}
V_p(c \cup cc_{\iota_0},\iota,D,\mathfrak a_c):=\sum_{j \in J} \sum_{\bm x \in R(c\cup cc_{\iota_0},\bm v_j)} V_p(\bm v_j,\bm x,\iota).
\end{align*}
When $D \subset F\otimes \mathbb R_+$, we also define $V_p(c,\iota,D,\mathfrak a_c)$ similarly.
\end{dfn}

\begin{dfn} 
When $D \subset F\otimes \mathbb R_{(n-1)+}$, we put 
\begin{align*}
X_p(c \cup cc_{\iota_0},\iota,D,\mathfrak a_c)
:=G_p(c \cup cc_{\iota_0},\iota,D,\mathfrak a_c)+W_p(c \cup cc_{\iota_0},\iota,D,\mathfrak a_c)+V_p(c \cup cc_{\iota_0},\iota,D,\mathfrak a_c).
\end{align*}
When $D \subset F\otimes \mathbb R_+$, we also define $X_p(c,\iota,D,\mathfrak a_c)$ similarly.
If $D$ is a Shintani domain, we put
\begin{align*}
X_p(c,\iota):=X_p(c,\iota,D,\mathfrak a_c), \ G_p(c,\iota):=G_p(c,\iota,D,\mathfrak a_c).
\end{align*}
\end{dfn}

\begin{prp} \label{welldef}
The definitions of $G_p,W_p,V_p,X_p$ doses not depend on the choice of the cone decomposition $D=\coprod_{j \in J} C(\bm v_j)$
when we fix $D$.
\end{prp}

\begin{proof}
We use a similar argument to the proof of \cite[Lemma 3.6]{Ka3}.
We can reduce the problem to a refinement of a cone: Let $D=C(\bm v)$ with $\bm v=(v_{1},\dots,v_{r}) \in \mathcal O_F^r$.
We abbreviate $\iota(\bm v)$ as $\bm v$. 
It suffices to show that $G_p,W_p,V_p$ does not change under the following operations.
\begin{enumerate}
\item[\I] Change the order of the basis $\bm v$.
\item[\II] Replace $v_{1}$ by $nv_{1}$ with $n \in \mathbb N$.
\item[\III] Decompose $C(\bm v)$ into $C(\bm v^\sharp)\coprod C(\bm v^\flat) \coprod C(\bm v^\natural)$ with 
\begin{align*}
\bm v^\sharp:=(v_{1},v_{1}+v_{2},v_3,\dots,v_{r}), \ 
\bm v^\flat:=(v_{1}+v_{2},v_{2},v_3,\dots,v_{r}), \ 
\bm v^\natural:=(v_{1}+v_{2},v_{3},\dots,v_{r}). 
\end{align*}
\end{enumerate}
The case of $V_p$ is completely the same as \cite[Lemma 3.6]{Ka3}.
The remaining cases follows from that
\begin{align*}
&\zeta_p(s,(v_{1},\dots,v_{r}),z) \text{ does not depend on the order of the basis }v_{1},\dots,v_{r}, \\
&\zeta_p(s,(v_{1},v_{2},\dots,v_{r}),z)=\sum_{k=0}^{n-1}\zeta_p(s,(nv_{1},v_{2},\dots,v_{r}),z+kv_1),  \\
&\zeta_p(s,\bm v,\bm x\ttt \bm v)=
\begin{cases}
\zeta_p(s,\bm v^\sharp,\bm x^\sharp\ttt \bm v^\sharp)+\zeta_p(s,\bm v^\flat,\bm x^\flat\ttt \bm v^\flat) & (x_1\neq  x_2), \\
\zeta_p(s,\bm v^\sharp,\bm x^\sharp\ttt \bm v^\sharp)+\zeta_p(s,\bm v^\flat,\bm x^\flat\ttt \bm v^\flat)
+\zeta_p(s,\bm v^\natural,\bm x^\natural\ttt \bm v^\natural)  & (x_1= x_2).
\end{cases}
\end{align*}
for $\bm x=(x_1,\dots,x_r) \in R(c,\bm v)$ or $\in R(c \cup cc_{\iota_0},\bm v)$. Here we put 
\begin{align*}
\bm x^\sharp&:=
\begin{cases}
(x_1-x_2+1,x_2,x_3,\dots,x_{r}) & \text{ if } x_1<x_2, \\
(x_1-x_2,x_2,x_3,\dots,x_{r}) & \text{ if } x_1>x_2, \\
(1,x,x_3,\dots,x_{r}) & \text{ if } x_1=x_2=:x, \\
\end{cases} \\
\bm x^\flat&:=
\begin{cases}
(x_1,x_2-x_1,x_3,\dots,x_{r}) & \text{ if } x_1<x_2, \\
(x_1,x_2-x_1+1,x_3,\dots,x_{r}) & \text{ if } x_1>x_2, \\
(x,1,x_3,\dots,x_{r}) & \text{ if } x_1=x_2=:x, \\
\end{cases} \\
\bm x^\natural&:=(x,x_3,x_4,\dots,x_r) \hspace{73.5pt} \text{ if } x_1=x_2=:x. 
\end{align*}
When all $v_i$ are positive, these equations follow from the $p$-adic interpolation property, since the same equations obviously hold for $\zeta(s,\bm v,z)$.
Furthermore $\zeta_p(s,\bm v,z)$ is continuous in the sense of Proposition \ref{prpofzp}-(ii), so we can generalize equations to all $\bm v$.
\end{proof}

We proved the following $p$-adic analogue of Theorem \ref{SYf}, although we will not use it in this paper.

\begin{thm}[{\cite[Theorem 6.2]{Ka1}, \cite[Theorem 3.1]{KY1}}] \label{Kf}
Assume that $\mathfrak p$ divides $\mathfrak f$ for any prime ideal $\mathfrak p$ lying above $p$.
If $p=2$, we further assume that $2\mathfrak p$ divides $\mathfrak f$ for any $\mathfrak p$ lying above $2$.
When $D$ is a Shintani domain, we have
\begin{align*}
\zeta_p'(0,c)=\sum_{\iota \in \HFR}X_p(c,\iota).
\end{align*}
Here $\zeta_p(s,c)$ is the $p$-adic interpolation function of $\zeta(s,c)$.
\end{thm}

\begin{lmm} \label{123p}
Let $\iota \in \HFR$ satisfy $\mfpi \mid \mathfrak f$. 
Then the same statements as in Lemmas \ref{123}, \ref{1232}, \ref{12} hold, by replacing $X,G,W,V,\log,\log_\iota$ with $X_p,G_p,W_p,V_p,\log_p,\log_{\iota,p}$, 
not excepting $\iota=\iota_0$.
We also replace $\zeta(0,\iota(\bm v_j),\iota(\bm x\ttt\bm v_j))$ with $\zeta_{\mathrm{fml}}(0,\iota(\bm v_j),\iota(\bm x\ttt\bm v_j))$
when $\iota(\bm v_j)$ has a negative entry.
\end{lmm}

\begin{proof}
The same proof as those of \cite[Chap.\ III, \S 3.6, 3.7]{Yo}, \cite[Lemmas 3.10, 3.11, 3.12]{Ka3} works,
by using Propositions \ref{prpofmg}-(ii), \ref{welldef}.
For example, consider the $p$-adic analogue of Lemma \ref{123}-(i)-\II:
Let $D=\coprod_{j \in J}C(\bm v_j)$, $D'=(\coprod_{j \in J-\{j_0\}}C(\bm v_j)) \coprod C(\epsilon \bm v_{j_0})$.
Then we have
\begin{align*}
&G_p(c,\iota,D',\mathfrak a_c)-G_p(c,\iota,D,\mathfrak a_c) \\
&=\sum_{\bm x \in R(c,\epsilon \bm v_{j_0})}L\Gamma_p(\iota(\bm x\ttt \epsilon \bm v_{j_0}),\iota(\epsilon \bm v_{j_0}))
-\sum_{\bm x \in R(c,\bm v_{j_0})}L\Gamma_p(\iota(\bm x\ttt\bm v_{j_0}),\iota(\bm v_{j_0})).
\end{align*}
We see that $R(c,\epsilon \bm v_{j_0})=R(c,\bm v_{j_0})$ by (\ref{eqofsets}).
Hence Proposition \ref{prpofmg}-(ii) states that 
\begin{align*}
G_p(c,\iota,D',\mathfrak a_c)-G_p(c,\iota,D,\mathfrak a_c)=-\iota(Z_{j_0})\log_p\iota(\epsilon)
\end{align*}
as desired.
The other cases can be proved similarly.
\end{proof}

\section{The case when a real place splits completely}

Theorem \ref{main1} is one of the main results in this paper.
The Archimedean part (\ref{main1arch}) was proved in \cite{Ka3}.
The $p$-adic part (\ref{main1padic}) is a new result, although it is proved by quite similar arguments.
Let $c_{\iota}=[(\nu_\iota)] \in C_\mathfrak f$ be as in Lemma \ref{ciota}.

\begin{thm} \label{main1}
Let $c \in C_\mathfrak f$, $\iota_0 \in \HFR$. 
Then there exist $u \in E_{F,+}$, $m \in \mathbb N$ satisfying 
\begin{align}
\exp(X(c,\iota))\exp(X(cc_{\iota_0},\iota))&=\iota(u)^{\frac{1}{m}} \quad (\iota \in \HFR,\ \iota \neq \iota_0), \label{main1arch} \\
X_p(c,\iota)+X_p(cc_{\iota_0},\iota)&=\frac{1}{m}\log_p \iota(u) \quad (\iota \in \HFR,\ \mfpi \mid \mathfrak f). \label{main1padic}
\end{align}
\end{thm}

\begin{proof}
(\ref{main1arch}) was proved in \cite[Proof of Theorem 3.1]{Ka3} by using Lemmas \ref{123}, \ref{1232}, \ref{12}, and \cite[Lemma 3.4]{Ka3}.
We can repeat the same argument for the $p$-adic analogue (\ref{main1padic}) by Lemma \ref{123p}:
Let $D$, $\nu$, $X_t$, $\epsilon_t$ ($t \in T$, $|T|<\infty$) be as in \cite[Lemma 3.4]{Ka3}.
Then we have 
\begin{itemize}
\item $D$ is a Shintani domain, $\nu \in F_{(n-1)+}$, $\epsilon_t \in E_{F,+}$. 
\item Each $X_t$ is a subset of $F\otimes \mathbb R_{(n-1)+}$ which can be expressed as a finite disjoint union of cones of $F$:
$X_t=\coprod_{j \in J_t} C(\bm v_j)$.
\item  We denotes by $\biguplus$ the multiset sum. Then we have
\begin{align*}
\left(D\coprod \nu D\right) \biguplus \left(\biguplus_{t \in T} \epsilon_t X_t \right)=\biguplus_{t \in T} X_t.
\end{align*}
\end{itemize}
Then, by Lemma \ref{1232}-(iv)-\II \ or its $p$-adic analogue in Lemma \ref{123p}, we have for $*=\emptyset$ or $p$
\begin{align*} 
X_*(c \cup cc_{\iota_0},\iota,D,\mathfrak a_c)+X_*(c \cup cc_{\iota_0},\iota,\nu D,\mathfrak a_c)
&=X_*(c \cup cc_{\iota_0},\iota,D\coprod \nu D,\mathfrak a_c) \\
&=\sum_{t \in T}\frac{\mathrm{Tr}_{F/\mathbb Q}(Z_t)}{n} \log_* \iota(\epsilon_t),
\end{align*}
where we put
\begin{align*}
Z_t:=\sum_{j \in J_t}\sum_{\bm x \in R(c \cup cc_{\iota_0},\bm v_j)}\zeta_{\mathrm{fml}}(0,\bm v_j,\bm x \ttt \bm v_j).
\end{align*}
First we assume that $c_{\iota_0}=[(1)]$ or $\mathcal O_{F,(n-1)+}^{\times}=\emptyset$.
Then we have
\begin{align*}
&X_*(c \cup cc_{\iota_0},\iota,D,\mathfrak a_c)=X_*(c,\iota,D,\mathfrak a_c), \\
&X_*(c \cup cc_{\iota_0},\iota,\nu D,\mathfrak a_c)\\
&=X_*(c \cup cc_{\iota_0},\iota,\nu_{\iota_0}^{-1}\nu D,(\nu_{\iota_0})\mathfrak a_c) 
- \frac{\zeta(0,cc_{\iota_0})}{h_{F,+}}\log_* \iota(u_{v_{\iota_0}}) \\
&=X_*(cc_{\iota_0},\iota,\nu_{\iota_0}^{-1}\nu D,(\nu_{\iota_0})\mathfrak a_c) - \frac{\zeta(0,cc_{\iota_0})}{h_{F,+}}\log_* \iota(u_{v_{\iota_0}}) \\
&=X_*(cc_{\iota_0},\iota,D,\mathfrak a_{cc_{\iota_0}}) - \frac{\zeta(0,cc_{\iota_0})}{h_{F,+}}\log_* \iota(u_{v_{\iota_0}}) + 
\frac{1}{N}\log_* \iota(u_{cc_{\iota_0},D,\nu_{\iota_0}^{-1}\nu D,\mathfrak a_{cc_{\iota_0}},(\nu_{\iota_0})\mathfrak a_c})
\end{align*}
by Lemma \ref{12}-(ii), (i), (ii), Lemma \ref{123}-(iv), and their $p$-adic analogues respectively.
Summarizing the above, we can write
\begin{align*}
&X_*(c,\iota,D,\mathfrak a_c)+X_*(cc_{\iota_0},\iota,D,\mathfrak a_{cc_{\iota_0}}) \notag \\
&=\frac{\zeta(0,cc_{\iota_0})}{h_{F,+}}\log_* \iota(u_{v_{\iota_0}})
-\frac{1}{N}\log_* \iota(u_{cc_{\iota_0},D,\nu_{\iota_0}^{-1}\nu D,\mathfrak a_{cc_{\iota_0}},(\nu_{\iota_0})\mathfrak a_c})
+\sum_{t \in T}\frac{\mathrm{Tr}_{F/\mathbb Q}(Z_t)}{n} \log_* \iota(\epsilon_t). 
\end{align*}
Next, we assume that $c_{\iota_0}\neq [(1)]$ and $\mathcal O_{F,(n-1)+}^{\times} \neq \emptyset$.
We take an element $\epsilon\in \mathcal O_{F,(n-1)+}^{\times}$.
Then we have
\begin{align*}
&X_*(c\cup cc_{\iota_0},\iota, D,\mathfrak a_c) \\
&=X_*(c,\iota,D,\mathfrak a_c)+X_*(cc_{\iota_0},\iota,\nu_{\iota_0}^{-1}\epsilon D,(\nu_{\iota_0})\mathfrak a_c) 
+\frac{\zeta(0,cc_{\iota_0})}{h_{F,+}}\log_* \iota(u_{\nu_{\iota_0}\epsilon^{-1}}), \\
&X_*(c \cup cc_{\iota_0},\iota,\nu D,\mathfrak a_c)
=X_*(c \cup cc_{\iota_0},\iota,\nu_{\iota_0}^{-1}\nu D,(\nu_{\iota_0})\mathfrak a_c) 
- \frac{\zeta(0,cc_{\iota_0})}{h_{F,+}}\log_* \iota(u_{v_{\iota_0}}), \\
&X_*(c \cup cc_{\iota_0},\iota,\nu_{\iota_0}^{-1}\nu D,(\nu_{\iota_0})\mathfrak a_c) \\
&=X_*(cc_{\iota_0},\iota,\nu_{\iota_0}^{-1}\nu D,(\nu_{\iota_0})\mathfrak a_c)
+X_*(c,\iota,\nu_{\iota_0}^{-2}\epsilon \nu D,(\nu_{\iota_0}^2)\mathfrak a_c)
+\frac{\zeta(0,c)}{h_{F,+}}\log_* \iota(u_{\nu_{\iota_0}\epsilon^{-1}}), \\
&X_*(cc_{\iota_0},\iota,\nu_{\iota_0}^{-1}\epsilon D,(\nu_{\iota_0})\mathfrak a_c)
=X_*(cc_{\iota_0},\iota,D,\mathfrak a_{cc_{\iota_0}}) 
+\frac{1}{N}\log_* \iota(u_{cc_{\iota_0},D,\nu_{\iota_0}^{-1}\epsilon D,\mathfrak a_{cc_{\iota_0}},(\nu_{\iota_0})\mathfrak a_c}), \\
&X_*(cc_{\iota_0},\iota,\nu_{\iota_0}^{-1}\nu D,(\nu_{\iota_0})\mathfrak a_c)
=X_*(cc_{\iota_0},\iota,D,\mathfrak a_{cc_{\iota_0}})
+\frac{1}{N}\log_* \iota(u_{cc_{\iota_0},D,\nu_{\iota_0}^{-1}\nu D,\mathfrak a_{cc_{\iota_0}},(\nu_{\iota_0})\mathfrak a_c}), \\
&X_*(c,\iota,\nu_{\iota_0}^{-2}\epsilon \nu D,(\nu_{\iota_0}^2)\mathfrak a_c)
=X_*(c,\iota,D,\mathfrak a_c)
+\frac{1}{N}\log_* \iota(u_{c,D,\nu_{\iota_0}^{-2}\epsilon \nu D,\mathfrak a_c,(\nu_{\iota_0}^2)\mathfrak a_c})
\end{align*}
by Lemma \ref{12}-(ii), (i), (ii), Lemma \ref{123}-(iv) (3 times), and their $p$-adic analogues respectively.
Summarizing the above, we can write
\begin{align*}
&X_*(c,\iota,D,\mathfrak a_c)+X_*(cc_{\iota_0},\iota,D,\mathfrak a_{cc_{\iota_0}}) \\ 
&=\frac{\zeta(0,cc_{\iota_0})}{2h_{F,+}}\log_* \iota(u_{v_{\iota_0}})
-\frac{1}{2N}\log_* \iota(u_{cc_{\iota_0},D,\nu_{\iota_0}^{-1}\epsilon D,\mathfrak a_{cc_{\iota_0}},(\nu_{\iota_0})\mathfrak a_c})
-\frac{\zeta(0,cc_{\iota_0})}{2h_{F,+}}\log_* \iota(u_{\nu_{\iota_0}\epsilon^{-1}}) \\
&\quad -\frac{1}{2N}\log_* \iota(u_{cc_{\iota_0},D,\nu_{\iota_0}^{-1}\nu D,\mathfrak a_{cc_{\iota_0}},(\nu_{\iota_0})\mathfrak a_c})
-\frac{1}{2N}\log_* \iota(u_{c,D,\nu_{\iota_0}^{-2}\epsilon \nu D,\mathfrak a_c,(\nu_{\iota_0}^2)\mathfrak a_c}) \\
&\quad -\frac{\zeta(0,c)}{2h_{F,+}}\log_* \iota(u_{\nu_{\iota_0}\epsilon^{-1}})
+\sum_{t \in T}\frac{\mathrm{Tr}_{F/\mathbb Q}(Z_t)}{2n} \log_* \iota(\epsilon_t).
\end{align*}
Thus the assertion holds in both cases. 
\end{proof}

\begin{crl} \label{crlofmain1}
Let $p,\iota$ satisfy $\mfpi \mid \mathfrak f$.
Let $\exp_p \colon \mathbb C_p \ra \mathbb C_p^\times$ be any group homomorphism 
which coincides with the usual power series $\sum_{n=0}^\infty \frac{z^n}{n!}$ on a neighborhood of $0$.
We denote by $\mu_\infty$ the group of all roots of unity.
\begin{enumerate}
\item The ratio $[\exp(X(c,\iota)):\exp_p(X_p(c,\iota))] \bmod \mu_\infty$ depends only on $c,\iota$. 
Strictly speaking, 
let $D,D'$ be Shintani domains, and $\mathfrak a_c,\mathfrak a_c'$ integral ideals satisfying $\mathfrak a_c\mathfrak f,\mathfrak a_c'\mathfrak f \in \pi(c)$.
Then we have
\begin{align*}
\frac{\exp(X(c,\iota,D,\mathfrak a_c))}{\exp(X(c,\iota,D',\mathfrak a_c'))} \equiv \frac{\exp_p(X_p(c,\iota,D,\mathfrak a_c))}{\exp_p(X_p(c,\iota,D',\mathfrak a_c'))}
\bmod \mu_\infty.
\end{align*}
\item Whenever $\mfpi \mid \mathfrak f$ ($p$ may vary), we have
\begin{align*}
\frac{\exp(X(c,\iota))\exp(X(cc_\iota,\iota))}{\exp_p(X_p(c,\iota))\exp_p(X_p(cc_\iota,\iota))} \equiv \exp(\zeta'(0,c))\exp(\zeta'(0,cc_\iota)) \bmod \mu_\infty.
\end{align*}
\end{enumerate}
\end{crl}

\begin{proof}
We see that $\exp_p\circ \log_p =\mathrm{id}$ on a neighborhood of $1$. Hence we have for a large enough $N$ 
\begin{align*}
(\exp_p( \log_p (z)))^{p^N}=\exp_p\circ \log_p (\langle z\rangle ^{p^N})=\langle z\rangle^{p^N} \quad (z \in \mathbb C_p^\times).
\end{align*}
In particular, when $\mathrm{ord}_p(z)=0$ we have
\begin{align*}
\exp_p( \log_p (z)) \equiv z \bmod \mu_\infty.
\end{align*}
Hence (i) follows from  Lemmas \ref{123}, \ref{123p}.
The assertion (ii) follows from Theorems \ref{SYf}, \ref{main1} since $\prod_{\iota \in \HFR} \iota(u)=N_{F/\mathbb Q}(u)=1$ for $u \in E_{F,+}$.
\end{proof}

\section{The case when a finite place splits completely}

In the previous section (especially in Corollary \ref{crlofmain1}-(ii)), $\exp(X(c,\iota))$ are the main terms
and $\exp_p(X_p(c,\iota))$ are the correction terms.
Their roles are exchanged in this section.

\subsection{A brief review of the results in \cite{KY1}}

Let $\mfpi$ be the prime ideal of $F$ corresponding to the $p$-adic topology of $F$ induced by $F \st{\iota}\hookrightarrow \mathbb C_p$
as in Definition \ref{piota}.

\begin{dfn}
\begin{enumerate}
\item We denote by $H_\mathfrak f$ the maximal ray class field modulo $\mathfrak f$ in the narrow sense,
by $\mathrm{Art}\colon C_\mathfrak f \st{\cong}\ra \mathrm{Gal}(H_\mathfrak f/F)$ the Artin map.
\item We denote the group of all characters of $C_{\mathfrak f}$ by $\hat C_{\mathfrak f}$.
Let $\chi \in \hat C_{\mathfrak f}$. For an integral ideal $\mathfrak g$ 
we denote the associated character $\in \hat C_{\mathfrak {fg}}$ by $\chi_{\mathfrak g}$.
Namely, $\chi_{\mathfrak g}$ is the composite map
\begin{align*}
C_{\mathfrak {fg}} \ra C_{\mathfrak f} \st{\chi}\ra \mathbb C^\times.
\end{align*}
\item Let $K$ be an extension field of $F$. For $\iota \in \HFR$, 
we take a lift $\tilde\iota \colon K \ra \mathbb C_p$ of $\iota\colon F \ra \mathbb C_p$ and put 
\begin{align*}
\mathfrak p_{K,\tilde\iota}:=\{z \in \mathcal O_K \mid |\tilde \iota(z)|_p<1 \}.
\end{align*}
Moreover we take a generator $\alpha_{K,\tilde\iota}$ of the principal ideal $\mathfrak p_{K,\tilde\iota}^{h_K}$, where $h_K$ is the class number of $K$.
\item We denote by $\overline{\mathbb{Q}}\log \overline{\mathbb{Q}}^{\times}$ 
(resp.\ $\overline{\mathbb{Q}}\log_{p} \overline{\mathbb{Q}}^{\times}$)
the $\overline{\mathbb{Q}}$-subspace of $\mathbb{C}$ (resp.\ $\mathbb{C}_{p}$)
generated by $\log a$ (resp.\ $\log_{p} a$) with $a\in \overline{\mathbb{Q}}^{\times}$.
For $\log a$, we take any branch of $\log$.
\item We define a $\overline{\mathbb{Q}}$-linear mapping $[\ ]_{p}$ by
\begin{align*}
[\ ]_{p}:\overline{\mathbb{Q}}\log\overline{\mathbb{Q}}^{\times}\rightarrow\overline{\mathbb{Q}}\log_{p}\overline{\mathbb{Q}}^{\times},
\quad \sum_{i=1}^k a_i\log b_i \mt \sum_{i=1}^k a_i\log_p b_i.
\end{align*}
This map is well-defined \cite[Lemma 5.1]{KY1} by a well-known theorem of A.~Baker.
\end{enumerate}
\end{dfn} 

\begin{rmk}
Let $H$ be an intermediate field of $H_\mathfrak f/F$, $\chi' $ a character of $\mathrm{Gal}(H/F)$.
Then there exists a character $\chi \in \hat C_\mathfrak f$ corresponding to $\chi'$ via the Artin map.
We note that $\chi$ may not be primitive.
The relation between the Artin $L$-function $L(s,\chi')$ and the Hecke $L$-function $L(s,\chi)$ can be written as
\begin{align*}
L(s,\chi)=L(s,\chi') \prod_{\mathfrak q \mid \mathfrak f, \ \mathfrak q \nmid \mathfrak f_{\chi'}} (1-\chi'(\mathrm{Frob}_\mathfrak q)N\mathfrak q^{-s}).
\end{align*}
Here $\mathfrak q$ runs over all prime ideal dividing $\mathfrak f$, not dividing the conductor $\mathfrak f_{\chi'}$ of $\chi'$.
We denote by $\mathrm{Frob}_\mathfrak q \in \mathrm{Gal}(H/F)$ the Frobenius automorphism.
\end{rmk}

\begin{prp}[{\cite[Lemma 5.5, Proposition 5.6]{KY1}}]
Let $\chi \in \hat C_{\mathfrak f}$. We assume that 
\begin{align*}
\mfpi\nmid \mathfrak f,\ \chi(\mfpi)=1.
\end{align*}
Then we have for $k \in \mathbb N$
\begin{align*}
\sum_{c \in C_{\fpi^k}} \chi_{\mfpi^k}(c) X(c,\iota) \in \overline{\mathbb{Q}}\log \overline{\mathbb{Q}}^{\times}.
\end{align*}
Moreover the quantity
\begin{align*}
\sum_{c \in C_{\fpi^k}} \chi_{\mfpi^k}(c) X_p(c,\iota)
-[\sum_{c \in C_{\fpi^k}} \chi_{\mfpi^k}(c) X(c,\iota)]_p \in \mathbb C_p
\end{align*}
does not depend on the choices of $D$ and $\mathfrak a_c$'s.
\end{prp}

\begin{proof}
The case $k=1$ follows from \cite[Lemma 5.5, Proposition 5.6]{KY1}.
We can use mathematical induction by \cite[Lemmas 5.3, 5.4]{KY1}.
\end{proof}

\begin{dfn}
Let $\chi \in \hat C_{\mathfrak f}$ satisfy $\mfpi\nmid \mathfrak f$, $\chi(\mfpi)=1$.
For $k \in \mathbb N$ we define 
\begin{align*}
Y_p(\chi_{\mfpi^k},\iota):=&\sum_{c \in C_{\fpi^k}} \chi_{\fpi^k}(c) X_p(c,\iota)
-[\sum_{c \in C_{\fpi^k}} \chi_{\fpi^k}(c) X(c,\iota)]_p.
\end{align*}
It follows that  
\begin{align*}
Y_p(\chi_{\mfpi^k},\iota)=\sum_{c \in C_{\fpi^k}} \chi_{\fpi^k}(c) G_p(c,\iota)
-[\sum_{c \in C_{\fpi^k}}\chi_{\fpi^k}(c) G(c,\iota)]_p
\end{align*}
since the $W,V$-terms are in $\overline{\mathbb{Q}}\log \overline{\mathbb{Q}}^{\times}$.
\end{dfn}

In \cite{KY1}, under the assumption $\chi(\mfpi)=1$, we formulated the following two conjectures on the exact value of $Y_p(\chi_{\mfpi},\iota)$, 
which is a refinement of Conjecture \ref{GSc}. 
The latter one has now become a corollary of Theorem \ref{main1}.

\begin{cnj}[{\cite[Conjecture A$'$]{KY1}}] \label{cnjky1}
Let $K$ be a CM-field which is abelian over $F$ with $\mathfrak f_{K/F}$ the conductor.
We assume that $\mfpi$ splits completely in $K/F$. (Hence $\mfpi \nmid \mathfrak f_{K/F}$.)
Then for any odd character $\chi'$ of $\mathrm{Gal}(K/F)$, we have
\begin{align*}
Y_p(\chi_{\mfpi},\iota)
=\frac{L(0,\chi')}{2h_K}\sum_{\sigma\in \mathrm{Gal}(K/F)}\chi'(\sigma)
\log_p\tilde\iota\left(\frac{\alpha_{K,\tilde\iota}^{\sigma\rho}}{\alpha_{K,\tilde\iota}^{\sigma}}\right).
\end{align*}
Here $\rho$ denotes the unique complex conjugation on $K$.
We take $\chi \in \hat C_{\mathfrak f_{K/F}}$ corresponding to $\chi'$.
\end{cnj}

\begin{thm}[{\cite[Conjecture 5.10]{KY1}}] \label{crl2ofmain1}
Let $\chi \in \hat C_{\mathfrak f}$.
We assume that $\chi$ is not totally odd and satisfies $\chi(\mfpi)=1$. (Hence $\mfpi \nmid \mathfrak f$.) Then we have
\begin{align*}
Y_p(\chi_{\mfpi},\iota)=0.
\end{align*}
\end{thm}

\begin{proof}
Since $\chi$ is not totally odd, there exists $\iota_0 \in \HFR$ satisfying $\chi(c_{\iota_0})=1$.
Then we can write
\begin{align*}
Y_p(\chi_{\mfpi},\iota)=\frac{1}{2} \sum_{c \in C_{\fpi}} \chi(c) \left(X_p(c,\iota)+X_p(cc_{\iota_0},\iota)-[X(c,\iota)+X(cc_{\iota_0},\iota)]_p\right).
\end{align*}
When $\iota \neq \iota_0$, we have $X_p(c,\iota)+X_p(cc_{\iota_0},\iota)-[X(c,\iota)+X(cc_{\iota_0},\iota)]_p=0$ by Theorem \ref{main1}.
Even when $\iota = \iota_0$, by Corollary \ref{crlofmain1} and $\chi(\mfpi)=1$, we have 
\begin{align*}
Y_p(\chi_{\mfpi},\iota)&=\frac{1}{2} \sum_{c \in C_{\fpi}} \chi(c) \log_p(\zeta'(0,c)\zeta'(0,cc_{\iota_0}))
=[\frac{d}{ds}L(s,\chi)(1-N\mfpi^{-s})]_p.
\end{align*}
Since $\chi$ is not totally odd, we see that $\mathrm{ord}_{s=0}L(s,\chi)\geq 1$, so $\frac{d}{ds}L(s,\chi)(1-N\mfpi^{-s})=0$.
Hence the assertion is clear.
\end{proof}

We reformulate Conjecture \ref{cnjky1} in the remaining of this subsection.

\begin{lmm}[{\cite[Lemma 6.4]{KY1}}] \label{ftofg}
Let $\chi \in \hat C_\mathfrak f$ satisfy $\chi(\mfpi)=1$. Then for an integral ideal $\mathfrak g$ of $F$ we have  
\begin{align*}
Y_p(\chi_{\mfpi\mathfrak g},\iota)
=\left(\prod_{\mathfrak q \mid \mathfrak g}(1-\chi_{\mfpi}(\mathfrak q))\right)Y_p(\chi_{\mfpi},\iota).
\end{align*}
Here $\mathfrak q$ runs over all prime ideals dividing $\mathfrak g$.
In particular for $k \in \mathbb N$ we have
\begin{align*}
Y_p(\chi_{\mfpi^k},\iota)=Y_p(\chi_{\mfpi},\iota).
\end{align*}
\end{lmm}

\begin{prp} \label{eqofcnjs}
Assume that $\mfpi \nmid \mathfrak f$.
Let $H$ be the fixed subfield of $H_\mathfrak f$ under $\langle\mathrm{Art}([\mfpi])\rangle$.
Conjecture \ref{cnjky1} is equivalent to each statement below.
\begin{enumerate}
\item Let $\chi \in \hat C_\mathfrak f$.
We assume that $\chi$ is totally odd and satisfies $\chi(\mfpi)=1$.
We denote the character of $\mathrm{Gal}(H/F)$ corresponding to $\chi$ by $\chi'$.
Then we have
\begin{align*}
Y_p(\chi_{\mfpi},\iota)
=\frac{-L(0,\chi)}{h_{H}}\sum_{\sigma\in \mathrm{Gal}(H/F)}\chi'(\sigma)
\log_p \tilde\iota\left(\alpha_{H,\tilde\iota}^{\sigma}\right).
\end{align*}
\item Let $\tau\in \mathrm{Gal}(H/F)$. We put 
\begin{align*}
Y_p(\tau,\iota):=\sum_{c \in C_{\fpi},\ \mathrm{Art}(c)=\tau} X_p(c,\iota)-[\sum_{c \in C_{\fpi},\ \mathrm{Art}(c)=\tau} X(c,\iota)]_p.
\end{align*}
Here $c$ runs over all ideal classes whose images under the composite map
$C_{\fpi} \ra C_\mathfrak f \ra \mathrm{Gal}(H_\mathfrak f/F) \ra \mathrm{Gal}(H/F)$ equal $\tau$.
Then we have
\begin{align*}
Y_p(\tau,\iota)=\frac{-1}{h_{H}}\sum_{c \in C_\mathfrak f}\zeta(0,c^{-1}) \log_{p}\tilde\iota\left(\alpha_{H,\tilde\iota}^{\tau\mathrm{Art}(c)}\right).
\end{align*}
\end{enumerate}
\end{prp}

\begin{proof}
First we note that $\chi \in \hat C_\mathfrak f$ is totally odd if and only if the intermediate field of $H_{\mathfrak{f}}/F$ 
corresponding to $\mathrm{Art}(\ker \chi)$ is a CM-field.
Hence the equivalence Conjecture \ref{cnjky1} $\LR$ (i) follows from Lemma \ref{ftofg}.
If $\chi$ is not totally odd, then we have $L(0,\chi)=0$, 
so the equation in (i) follows from Theorem \ref{crl2ofmain1}.
Therefore the equivalence (i) $\LR$ (ii) follows from the orthogonality of characters.
\end{proof}

\begin{rmk} \label{final}
Assume that $\mathfrak p_\iota \mid \mathfrak f$.
Let $H$ be an intermediate field of $H_\mathfrak f/F$.
First we assume that the real place $\iota$ splits completely in $H$.
Then, by Remark \ref{sccase} and Corollary \ref{crlofmain1}-(ii), we see that the Stark conjecture (\ref{Su}) with $v=\iota$ implies 
\begin{align*}
\left(\prod_{c \in \phi_H^{-1}(\tau)}\frac{\exp(X(c,\iota))}{\exp_p(X_p(c ,\iota))} \right)^2
\in \tilde\iota(\mathcal O_H^\times) \quad (\tau \in \mathrm{Gal}(H/F)),
\end{align*}
where $\phi_H$ denotes the composite map
\begin{align*}
C_\mathfrak f \ra \mathrm{Gal}(H_\mathfrak f/F) \ra \mathrm{Gal}(H/F).
\end{align*}
Strictly speaking, we put $\prod \frac{\exp_p(X_p(c ,\iota))}{\exp(X(c,\iota))}:=\frac{\prod \exp_p(X_p(c,\iota))}{\prod\exp(X(c,\iota))}$.
Next we assume that $\mfpi$ splits completely in $H$.
Then Proposition \ref{eqofcnjs}-(ii) states that Conjecture \ref{cnjky1} implies
\begin{align*}
\left(\prod_{c \in \phi_H^{-1}(\tau)}\frac{\exp(X(c,\iota))}{\exp_p(X_p(c ,\iota))} \right)^{h_HW}
\in \tilde\iota(\alpha) \ker \log_p \quad (\tau \in \mathrm{Gal}(H/F)).
\end{align*}
Here $W$ is the least common multiple of the denominators of $\zeta(0,c^{-1})'s$, $\alpha$ is a $\mfpi$-unit of $H$ defined as 
\begin{align*}
\alpha:=\prod_{c \in C_{\mathfrak f_0}}\alpha_{H,\tilde\iota}^{\tau\mathrm{Art}_0(c)W\zeta(0,c^{-1})},
\end{align*}
where we put $\mathfrak f_0$ to be the prime-to-$\mathfrak p_\iota$ part of $\mathfrak f$ and 
$\mathrm{Art}_0\colon C_{\mathfrak f_0} \ra \mathrm{Gal}(H/F)$ denotes the Artin map.
We note that $\ker \log_p$ is generated by rational powers of $p$ and the roots of unity.
\end{rmk}

\end{document}